\documentclass{article}
\usepackage[utf8]{inputenc}
\usepackage[margin=1in]{geometry}
\usepackage{amsmath} 
\usepackage{amsthm}
\usepackage{todonotes}
\usepackage[nolist]{acronym}
\usepackage{xcolor}
\usepackage{amsmath}
\usepackage{amssymb}
\usepackage{mathtools}
\usepackage{subcaption}
\usepackage{hyperref}
\hypersetup{colorlinks=true, linkcolor=black}
\usepackage[ruled]{algorithm2e}
\usepackage{mathrsfs}

\definecolor{dkgreen}{rgb}{0,0.4,0}
\definecolor{dkred}{rgb}{0.7,0,0}

\newcommand{\R}{\mathbb{R}} 
\newcommand{\K}{\mathbb{K}} 
 
\newcommand{\N}{\mathbb{N}}

\newcommand{\Lie}{\mathcal{L}}

\newcommand{\M}{\mathbb{M}}

\newcommand{\A}{\mathcal{A}}

\newcommand{\bm}{\mathbf{m}}

\DeclarePairedDelimiter{\abs}{\lvert}{\rvert}
\DeclarePairedDelimiter{\norm}{\lVert}{\rVert}
\DeclarePairedDelimiter{\ceil}{\lceil}{\rceil}
\DeclarePairedDelimiter{\floor}{\lfloor}{\rfloor}

\DeclarePairedDelimiterX{\inp}[2]{\langle}{\rangle}{#1, #2}

\DeclarePairedDelimiter{\Mp}{\mathcal{M}_+(}{)}

\DeclareMathOperator*{\argmin}{arg\!\,min}

\newtheorem{thm}{Theorem}[section]
\newtheorem{lem}[thm]{Lemma}

\newtheorem{prob}[thm]{Problem}

\newtheorem{defn}{Definition}[section]
\newtheorem{conj}{Conjecture}[section]

\newtheorem{rmk}{Remark} 
\definecolor{cn}{RGB}{93,147,191}           
\definecolor{cz}{RGB}{233,  72, 73}      
\definecolor{cp}{RGB}{113, 191, 110}    

\newcommand{\hs}{\mathcal{H}}

\newcommand{\bmu}{\bar{\mu}}

\newcommand{\ubar}[1]{\text{\b{$#1$}}} 

\title{\LARGE \bf Peak Estimation of Time Delay Systems \\ using Occupation Measures
}

\renewcommand\footnotemark{}

\author{Jared Miller$^1$, Milan Korda $^{2,3}$, Victor Magron $^{2,4}$,  Mario Sznaier$^1$
\thanks{$^1$J. Miller, and M. Sznaier are with the Robust Systems Lab,  ECE Department, Northeastern University, Boston, MA 02115. (e-mails: miller.jare@northeastern.edu, msznaier@coe.neu.edu).}
\thanks{$^2$Polynomial OPtimization  Team, LAAS-CNRS, Toulouse, 31400, France}
\thanks{$^3$Faculty of Electrical Engineering, Czech Technical University in Prague, Technick\'a 2, CZ-16626 Prague, Czech Republic}
\thanks{$^4$Institut de Math\'{e}matiques de Toulouse, 31062, France}
\thanks{J. Miller and M. Sznaier were partially supported by NSF grants   ECCS--1808381 and CNS--2038493, AFOSR grant FA9550-19-1-0005, and ONR grant N00014-21-1-2431.  
J. Miller was in part supported by the Chateaubriand Fellowship (performed at LAAS-CNRS) of the Office for Science \& Technology of the Embassy of France in the United States, and by the International Student Exchange Program from AFOSR.}
\thanks{The work of  M. Korda and V. Magron was partly supported by the European Union's Horizon 2020 research and innovation programme under the Marie Sklodowska-Curie Actions, grant agreement 813211 (POEMA). The work of M. Korda was also partly supported by the Czech Science Foundation (GACR) under contract No. 20-11626Y, and by the AI Interdisciplinary Institute (ANITI) funding, through the French "Investing for the Future PIA3" program under the Grant  agreement ANR-19-PI3A-0004.
}}

\begin{document}

\maketitle


\begin{abstract}
This work proposes a method to compute the maximum value obtained by a state function along trajectories of \iac{DDE}. An example of this task is finding the maximum number of infected people in an epidemic model with a nonzero incubation period. The variables of this peak estimation problem include the stopping time and the original history (restricted to a class of admissible histories). The original nonconvex \ac{DDE} peak estimation problem is approximated by an infinite-dimensional \ac{LP} in occupation measures, inspired by existing measure-based methods in peak estimation and optimal control.
This \ac{LP} is approximated from above by a sequence of \acp{SDP} through the moment-\ac{SOS} hierarchy. Effectiveness of this scheme in providing peak estimates for \acp{DDE} is demonstrated with provided examples.
\acresetall
\end{abstract}


\section{Introduction}
\label{sec:introduction}

This paper presents an algorithm to upper bound extreme values of a state function attained along trajectories of a \ac{DDE}. 
 The dynamics of a \ac{DDE} depend on a history of the state, in contrast to an \ac{ODE} in which the dynamics are a function only of the present values of state \cite{hale1971functional,kuang1993delay,bellen2013numerical,fridman2014intro}. This paper will involve analysis of \acp{DDE} in a state space $X \subset \R^n$ over a time horizon $T < \infty$ with a single fixed discrete bounded delay $\tau \in (0, T)$.

 Trajectory evolution of \iac{DDE}  depends on an initial history $x_h: [-\tau, 0] \rightarrow X$ rather than simply an initial condition $x_0 \in X$ for a corresponding \ac{ODE}. The  evaluation at time $t$ for a trajectory starting with a history $x_h$ will be denoted as $x(t \mid x_h)$. 
 A function class $\mathcal{H}$ of histories may be defined, allowing for the definition of differential inclusions of \acp{DDE}. 
A peak estimation problem may be defined on a time-delay system to find the maximum value of a state function $p$ along system trajectories given a class of initial histories $\mathcal{H}$ as
\begin{subequations}
\label{eq:peak_delay_traj}
    \begin{align}
    P^* = & \sup_{t^* \in [0, T], \; x_h(\cdot)} p(x(t^* \mid x_h)) &\\
    & \dot{x} =  f(t, x(t), x(t-\tau)) & & \forall t \in [0, T] \label{eq:delay_dynamics} \\
    & x(t) = x_h(t) & & \forall t \in [-\tau, 0]\\
     & x_h(\cdot) \in \mathcal{H}.
    \end{align}
\end{subequations}

The variables in Problem \eqref{eq:peak_delay_traj} are the stopping time $t^*$ and the initial history $x_h$.
Problem \eqref{eq:peak_delay_traj} is a \ac{DDE} version of the (generically nonconvex) \ac{ODE} peak estimation program studied in \cite{cho2002linear, fantuzzi2020bounding}. The peak estimation task in \eqref{eq:peak_delay_traj} is an instance of \iac{DDE} \ac{OCP} with a free terminal time and a zero running (integrated) cost.

This work uses measure-theoretic methods in order to provide certifiable upper bounds on the peak value $P^*$ from \eqref{eq:peak_delay_traj}. The first application of measure-theoretic methods towards \acp{DDE} was in \cite{warga1974optimal}, in which the control input was relaxed into a Young Measure \cite{young1942generalized} (probability distribution at each point in time) \cite{warga2014optimal}. This Young-Measure-based relaxed control yields the \ac{OCP} optimal value in the case of a single discrete time delay under convexity, regularity, and compactness assumptions. However, the Young Measure control programs may result in a lower bound when there are two or more delays in the system dynamics (there exist Young-Measure solutions that do not correspond to \ac{OCP} solutions) \cite{rosenblueth1991relaxation, rosenblueth1992strongly}. Adding new measures and constraints allows for the construction of tight Young Measure \ac{OCP} approximations 
at the cost of significantly more complicated programs \cite{rosenblueth1992proper}. 

Occupation measures are nonnegative Borel measures that contain all possible information about trajectory behavior, and are a step beyond Young Measures in terms of abstraction and relaxation. The work of \cite{lewis1980relaxation} proves that a convex infinite-dimensional \ac{LP} in occupation measures for an \ac{ODE} \ac{OCP} has the same optimal value as the original \ac{OCP} under compactness, convexity, and regularity conditions. The problem of estimation of the peak of the expected value of a given state function for stochastic processes may be solved using occupation measures under these same conditions \cite{cho2002linear}. The Moment-\ac{SOS} hierarchy offers a sequence of outer approximations (lower bounds on \ac{OCP}/upper bounds on peak estimates) as found through solving \acp{SDP} of increasing size \cite{lasserre2009moments}. The moment-\ac{SOS} hierarchy has been applied to dynamical problems including barrier functions \cite{prajna2004safety}, \acp{OCP}  \cite{henrion2008nonlinear,papa2009sosdelay}, peak estimation \cite{fantuzzi2020bounding, miller2020recovery}, region of attraction estimation \cite{korda2013inner}, reachable set estimation \cite{magron2017discrete}  and distance estimation \cite{miller2022distance_short}. 

Use of the moment-\ac{SOS} hierarchy towards analysis of \acp{DDE} includes finding stability and safety certificates \cite{papachristodoulou2005tutorial, prajna2005methods, papa2009sosdelay}.
Prior work on using occupation measures for problems in time delays includes \ac{ODE}-PDE models in \cite{marx2018entropy, korda2018momentspde}, a Riesz-frame system in \cite{magron2020optimal}, and a gridded \ac{LP} framework for optimal control given a single history $x_h$ in \cite{barati2012optimal}. Peak estimation has been conducted on specific time-delay systems such as the forced Li\'{e}nard model \cite{suresh2018forced} and compartmental epidemic models \cite{sadeghi2021universal}.

The contributions of this paper are:
\begin{itemize}
    \item A theory of \ac{MV}-solutions to \acp{DDE} with multiple histories (in $\hs$) and free terminal time
    \item A measure \ac{LP} that upper-bounds problem \eqref{eq:peak_delay_traj}
    \item A convergent sequence of \acp{LMI} (and resultant \acp{SDP}) to the measure upper-bound
\end{itemize}

This paper is organized as follows: Section \ref{sec:preliminaries} formalizes notation and summarizes concepts in measure theory, time-delay, occupation measures, and \ac{ODE} peak estimation. Section \ref{sec:peak_lp} defines \ac{MV}-solutions for free-terminal-time \ac{DDE} trajectories to create a primal-dual pair of \acp{LP} to upper-bound \eqref{eq:peak_delay_traj}.
Section \ref{sec:moment} applies the Moment-\ac{SOS} hierarchy towards generating \acp{SDP} to upper-bound the peak-estimation measure \ac{LP}. 
Section \ref{sec:delay_examples} provides examples of \ac{DDE} peak estimation. Section \ref{sec:extensions_delay} extends the \acp{DDE} peak framework by allowing for distance estimation, shaping constraints on histories, and uncertainty. Section \ref{sec:conclusion} concludes the paper. Appendix  \ref{app:delay_structure} extends the \ac{MV}-solution framework towards continuous-time systems with proportional delays and discrete-time systems with long time delays.  Appendix \ref{app:strong_duality_delay} performs the proof of strong duality for the \acp{DDE} peak estimation \acp{LP}. Appendix \ref{app:ocp} finds and analyzes structural properties of \ac{DDE} \acp{OCP} subvalue functionals. Appendix \ref{app:more_accurate} reduces conservatism of \ac{OCP} approximations by performing spatio-temporal partitioning and applying double-integral subvalue functionals. Appendix \ref{app:joint_component} introduces a more conservative but computationally simpler notion of \ac{MV}-solution for \acp{DDE}.

\section{Preliminaries}
\label{sec:preliminaries}
\begin{acronym}

\acro{BSA}{Basic Semialgebraic}

\acro{DDE}{Delay Differential Equation}

\acro{HJB}{Hamilton-Jacobi-Bellman}

\acro{LMI}{Linear Matrix Inequality}
\acroindefinite{LMI}{an}{a}
\acroplural{LMI}[LMIs]{Linear Matrix Inequalities}

\acro{LP}{Linear Program}
\acroindefinite{LP}{an}{a}

\acro{MV}{Measure-Valued}

\acro{OCP}{Optimal Control Problem}
\acroindefinite{OCP}{an}{a}

\acro{ODE}{Ordinary Differential Equation}
\acroindefinite{ODE}{an}{a}

\acro{PC}{Piecewise Continuous}


\acro{PSD}{Positive Semidefinite}

\acro{SDP}{Semidefinite Program}
\acroindefinite{SDP}{an}{a}


\acro{SOS}{Sum of Squares}
\acroindefinite{SOS}{an}{a}

\end{acronym}

\subsection{Notation}

The $n$-dimensional real Euclidean vector space is $\R^n$. The set of natural numbers is $\N$, and the set of $n$-dimensional multi-indices is $\N^n$. The degree of a multi-index $\alpha \in \N^n$ is $\abs{\alpha} = \sum_{i=1}^n \alpha_i$.
The set of polynomials with real coefficients in an indeterminate $x$ is $\R[x]$. Each polynomial $p(x) \in \R[x]$ has a unique representation in terms of a finite index set $\mathcal{J} \subset \N^n$ and coefficients $\{p_\alpha\}_{\alpha \in \mathcal{J}}$ with $p_\alpha \neq 0$ as $p(x) = \sum_{\alpha \in \mathcal{J}} p_\alpha \left(\prod_{i=1}^n x_i^{\alpha_{i}}\right) = \sum_{\alpha \in \mathcal{J}} p_\alpha x^\alpha$. The degree of a polynomial $\deg p(x)$ is equal to $\max_{\alpha \in \mathcal{J}} \abs{\alpha_j}$. The subset of polynomials with degree at most $d$ is $\R[x]_{\leq d} \subset \R[x]$.

\subsection{Analysis and Measure Theory}
Let $X$ be a topological space.
The set of continuous functions over a space $X$ is $C(X)$, and its subcone of nonnegative functions over $X$ is $C_+(X)$. The subset of once-differentiable functions over $X$ is $C^1(X) \subset C(X)$. A single-variable function $g(t)$ is \ac{PC}  over the domain $[a, b]$ if there exist $B \in \N \backslash \{0\}$ and a finite number of time-breaks $t_0 = a < t_1 < t_2 <  \cdots < t_{B} < b = t_{B+1}$ such that the function $g(t)$ is continuous in each interval $[t_k, t_{k+1})$ for $k = 0..B$.
The class of \ac{PC} functions from the time interval $[-\tau, 0]$ to $X$ is $PC([-\tau, 0], X)$.

The set of nonnegative Borel measures over $X$ is $\Mp{X}$. A pairing exists between functions $p \in C(X)$ and measures $\mu \in \Mp{X}$ by Lebesgue integration with $\inp{p}{\mu} = \int_X p(x) d \mu(x)$. 
This pairing is a duality pairing and defines an inner product 
between $C_+(X)$ and $\Mp{X}$ when $X$ is compact. The $\mu$-measure of a set $A \subseteq X$ may be defined in terms of $A$'s indicator function ($I_A(x)=1$ if $x \in A$ and $I_A(x)=0$ otherwise) as $\mu(A) = \inp{I_A(x)}{\mu(x)}$. The quantity $\mu(X)$ is called the mass of $\mu$, and $\mu$ is a probability distribution if $\mu(X) = 1$. The support of $\mu$ is the set of all points $x$ such that all open neighborhoods $N_x \ni x$ satisfy $\mu(N_x) > 0$. 
Two special measures are the Dirac delta and the Lebesgue measure. The Dirac delta $\delta_{x}$ with respect to a point $x \in X$ obeys the point-evaluation pairing $\inp{p}{\delta_{x}} = p(x)$ for all $p \in C(X)$. The Lebesgue (volume) distribution has the definition $\inp{p}{\lambda_X} = \int_X p(x) dx$. Further details about measure theory are available in \cite{tao2011introduction}.

Given spaces $X$ and $Y$, the projection $\pi^x: X \times Y \rightarrow X$ is the map $(x, y) \mapsto x$.
For measures $\mu \in \Mp{X}$ and $\nu \in \Mp{Y}$, the product measure $\mu \otimes \nu \in \Mp{X \times Y}$ is the unique measure satisfying $(\mu\otimes \nu)(A \times B) = \mu(A)\nu(B)$ for all subsets $A \subseteq X, \ B \subseteq Y$. For two measures $\mu, \xi \in \Mp{X}$, the measure $\mu$ dominates $\xi$ ($\xi \leq \mu$) if $\xi(A) \leq \mu(A), \ \forall A \subset X$. To every dominated measure $\xi \leq \mu$, there exists a slack measure $\hat{\xi} \in \Mp{X}$ such that 
$\xi + \hat{\xi} = \mu$.


The pushforward of a map $Q: X \rightarrow Y$ along a measure $\mu$ is $Q_\# \mu$, with the relation $\inp{z}{Q_\# \mu} = \inp{z \circ Q }{\mu}$ holding for all $z \in C(Y)$. Given $\eta \in \Mp{X \times Y}$, the projection-pushforward $\pi^x_\# \eta$ is the $x$-marginalization of $\eta$. The pairing of $p \in C(X)$ with $\pi^x_\#\eta$ may be equivalently expressed as $\inp{p}{\pi^x_\#\eta} = \inp{p}{\eta}$.
The adjoint of a linear map $\Lie: C(X)\rightarrow C(Y)$ is a mapping $\Lie^\dagger:M(Y)\to M(X)$ satisfying $\inp{\Lie p}{\nu} = \inp{p}{\Lie^\dagger \nu}$ for all $p \in C(X)$ and $\nu \in M(Y)$.


\subsection{Time Delay Systems}

A single-variable function $g(t)$ is \ac{PC}  over the domain $[a, b]$ if there exist $B \in \N \backslash \{0\}$ and a finite number of time-breaks $t_0 = a < t_1 < t_2 <  \cdots < t_{B} < b = t_{B+1}$ such that the function $g(t)$ is continuous in each interval $[t_k, t_{k+1})$ for $k = 0..B$.
The class of \ac{PC} functions from the time interval $[-\tau, 0]$ to $X$ is $PC([-\tau, 0], X)$.

Given a \ac{PC} state history $t \mapsto x_h(t), \ t \in [-\tau, 0]$,  a unique forward trajectory $x(t \mid x_h)$ of \eqref{eq:delay_dynamics} exists on  $t \in [0, T]$ if the function $(t,x_0,x_1) \mapsto f(t, x_0, x_1)$ is locally Lipschitz in all variables. 


Trajectories of time-delay systems with the form of \eqref{eq:delay_dynamics} with $f$ locally Lipschitz satisfy a smoothing property as shown in Figure \ref{fig:increasing_continuity}. The order of derivatives that are continuous will increase by 1 every $\tau$ time steps \cite{fridman2014intro}. An example of such a time-delay system with increasing continuity  is visualized in Figure \ref{fig:increasing_continuity} with system dynamics
\begin{equation}
    x'(t) = -2x(t)-2x(t-1). \label{eq:time_delay_fig}
\end{equation} 

\begin{figure}[h]
    \centering
    \includegraphics[width=0.7\linewidth]{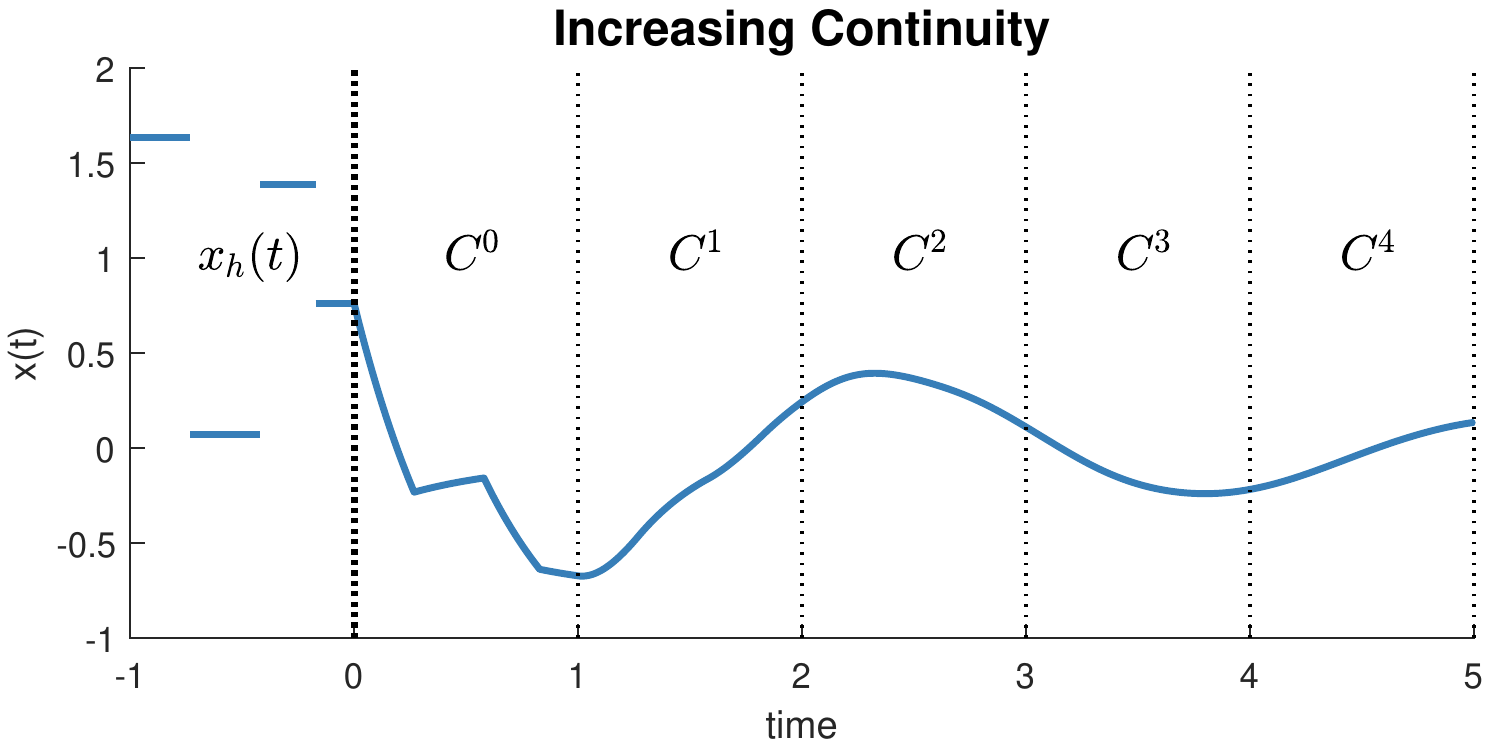}
    \caption{Continuity of \eqref{eq:time_delay_fig} trajectories increases every $\tau=1$ time step}
    \label{fig:increasing_continuity}
\end{figure}

Figure \ref{fig:same_initial} plots multiple trajectories of \eqref{eq:time_delay_fig} whose histories are lines passing through $x_h(0)=1$, but whose evolution after time $t = 0$ is different.

\begin{figure}[h]
    \centering
    \includegraphics[width=0.7\linewidth]{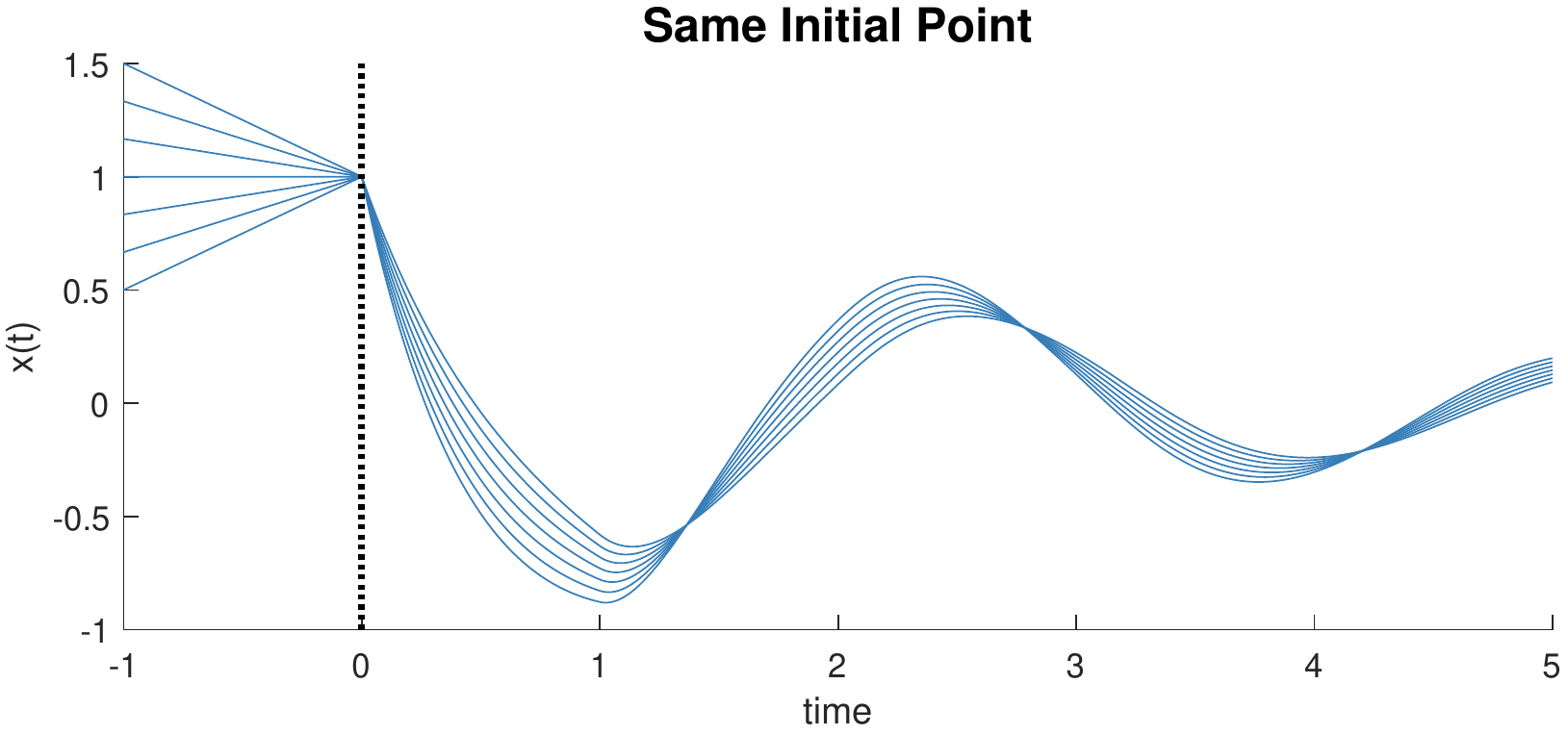}
    \caption{All histories of Figure \eqref{eq:time_delay_fig} pass through $x_h(0)=1$}
    \label{fig:same_initial}
\end{figure}

The behavior of time-delay systems may change and bifurcate as the time delays change. A well-studied example of $\dot{x} = -x(t-\tau)$ 
is plotted in Figure 
\eqref{fig:stability_switch} \cite{fridman2014intro}, 
in which the system is stable (to $x=0$) for all bounded \ac{PC} histories with $\tau \in [0, \pi/2)$, has bounded oscillations for some initial histories at $\tau=\pi/2$ (e.g., constant $x_h$ in time), and is unstable (divergent oscillations to $\pm \infty$) for all similar histories with $\tau > \pi/2$ \cite{fridman2014intro}.

\begin{figure}[h]
    \centering
    \includegraphics[width=0.7\linewidth]{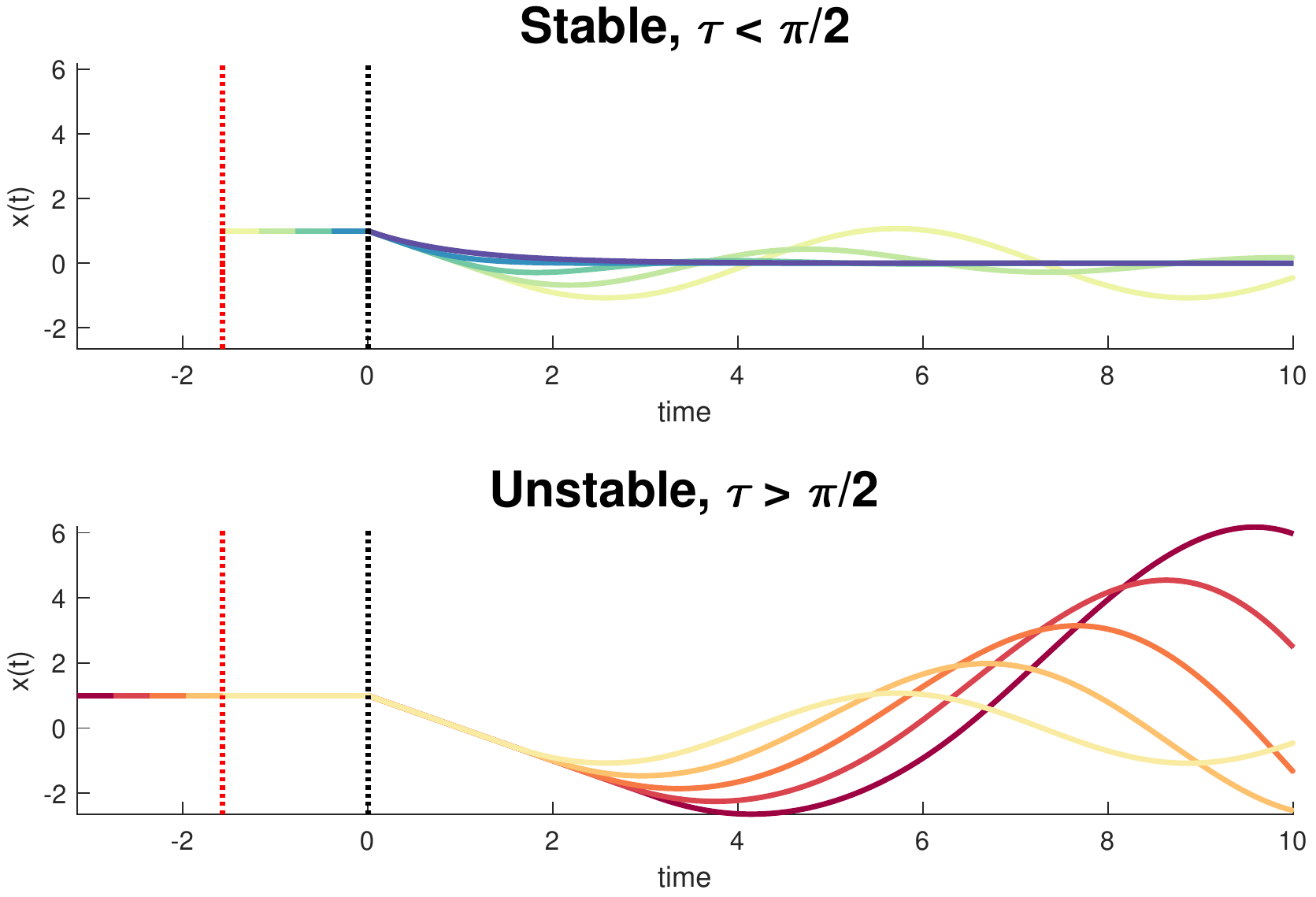}
    \caption{Bifurcation of stability as $\tau$ exceeds $\pi/2$ in $\dot{x}=-x(t-\tau)$}
    \label{fig:stability_switch}
\end{figure}

Problem \eqref{eq:peak_delay_traj} involves a class of histories $\hs$. In this paper, we will impose that $\hs$ is graph-constrained:
\begin{defn}
\label{defn:graph_constrained}
    The history class $\hs$ is \textbf{graph-constrained} if $\hs$ is the set of histories whose graph lies within a given set $H_0 \subseteq [-\tau, 0] \times X$,
\begin{align*}
    \hs = \{x_h \in PC([-\tau, 0], X) \mid (t, x_h(t)) \in H_0 \ \forall t \in [-\tau, 0]\},
\end{align*} 
and there are no other continuity restrictions on histories.
\end{defn}

\subsection{Occupation Measures}

The \textbf{occupation measure} associated with an interval $[a, b]  \subset \R$ and a curve $t \mapsto x(t) \in PC([a, b], X)$ is the pushforward of the Lebesgue distribution (in time) $\lambda_{[a, b]}$ along the curve evaluation. 
Such an occupation measure $\mu_{x(\cdot)} \in \Mp{[a, b] \times X}$ satisfies a relation for all $v \in C([a, b] \times X):$
\begin{align}
    \inp{v}{\mu_{x(\cdot)}} & = \textstyle \int_{a}^b v(t, x(t)) dt.
\end{align}


Occupation measures can be extended to controlled dynamics.
Let $U \subset \R^m$ be a set of input-values and define the following controlled dynamics (in which $u(t) \in U$ for all $t \in [0, T]$) :
\begin{align}
    \dot{x}(t) &= f(t, x(t), u(t)). \label{eq:dynamics_u}
\end{align}

The  occupation measure of a trajectory of \eqref{eq:dynamics_u} given a stopping time $t^*$, an initial condition $x_0 \in X_0 \subset X$ and a measurable  control $u(\cdot)$ (such that $u(t)$ is a probability distribution over $U$ for each $t \in [0, t^*]$) for sets $A \subseteq [0, T], \ B \subseteq X, \ C \subseteq U$, is
\begin{align}   
\label{eq:occ_free_single}
    &\mu(A \times B\times C \mid t^*, x_0) = \\
    &\int_{[0, t^*] \times X_0\times U} I_{A \times B \times C}\left((t, x(t \mid x_0,  u(\cdot)), u(t)\right) dt. \nonumber
\end{align}

The occupation measure in \eqref{eq:occ_free_single} may be averaged over a distribution of initial conditions  $\mu_0(x_0)$:
\begin{align}
    \label{eq:avg_free_occ}
    &\mu(A \times B\times C \mid t^*) = \\
    &\int_{X_0} \mu(A \times B\times C \mid t^*, x_0) d\mu_0(x_0). \nonumber
\end{align}




A linear operator $\Lie_f$ may be defined for every $v \in C^1([-\tau,T]\times \R^n)$ by
\begin{equation}
\label{eq:lie}
    \Lie_f v(t, x) = \partial_t v(t,x) + f(t,x, u) \cdot \nabla_x v(t,x).
\end{equation}
A distribution of initial conditions $\mu_0 \in \Mp{X_0}$, free-terminal-time values $\mu_p \in \Mp{[0, T] \times X}$, and occupation measures $\mu \in \Mp{[0, T] \times X \times U}$ from \eqref{eq:avg_free_occ} are connected together by Liouville's equation for all $v \in C^1([0, T] \times X)$:
\begin{subequations}
\label{eq:liou}
\begin{align}
\inp{v}{\mu_p} &= \inp{v(0, x)}{\mu_0(x)} + \inp{\Lie_f v}{\mu} \label{eq:liou_int}\\
\mu_p &= \delta_0 \otimes \mu_0 + \pi^{t x}_\# \Lie_f^\dagger \mu. \label{eq:liou_weak}
\end{align}
\end{subequations}
Equation \eqref{eq:liou_weak} is a shorthand notation for \eqref{eq:liou_int} when applied to all $C^1$  functions $v$. Note that the $\pi^{t x}_\#$ marginalizes out the input $u$ in the occupation measure $\mu$. Any $\mu$ as part of a tuple of measures $(\mu_0, \mu_p, \mu)$ satisfying \eqref{eq:liou}
is referred to as a \textbf{relaxed occupation measure}.

\section{Peak Measure Program}
\label{sec:peak_lp}
This section will formulate a measure-valued \ac{LP} which upper-bounds Problem \eqref{eq:peak_delay_traj} in objective.

\subsection{Assumptions}

The following assumptions will be imposed on the peak estimation Problem \eqref{eq:peak_delay_traj}:
\begin{enumerate}
    \item[A1] The set $[-\tau, T] \times X$ is compact with $\tau<T$;
    \item[A2] The function $f$ is Lipschitz inside $[0, T] \times X^2$;
    \item[A3] Any trajectory $x(\cdot \mid x_h)$ with $x_h \in \hs$ such that $x(t\mid x_h) \not\in X$ for some $t \in [0, T]$ also satisfies $x(t'\mid x_h) \not\in X$ for all $t' \geq t$; 
    \item[A4] The objective $p$ is continuous;
    \item[A5] The history class $\hs$ is graph-constrained by $H_0 \subset [-\tau, 0] \times X$.
\end{enumerate}

In the case where $\tau > T$, the delayed state $t \mapsto x(t-\tau)$ is fully specified in time $[0, T]$ without requiring dynamics information, and \eqref{eq:peak_delay_traj} reduces to a peak estimation problem over \acp{ODE}.
All tracked histories in $\hs$ are bounded due to assumption A1 (since the range $X$ is compact). The nonreturn assumption A3 ensures that a trajectory cannot leave and then return to $X$ to produce a lower value of $p$, given that the occupation-measure-based techniques used in this paper can only track trajectories while they are in $X$.

\subsection{Measure-Valued Solution}

The initial set $X_0$ is the $t=0^+$ slice of $H_0$. Equation  \eqref{eq:weak_solution} describes the measures $(\mu_h, \mu_0, \mu_p, \bmu_0, \bmu_1, \nu)$ that will be used to form a free-terminal-time \ac{MV}-solution to the \ac{DDE} \eqref{eq:delay_dynamics} with multiple histories (in $\hs$):
\begin{subequations}
\label{eq:weak_solution}
    \begin{align}
        &\textrm{History}  & \mu_h &\in \Mp{H_0} \label{eq:weak_solution_history} \\
       & \textrm{Initial} & \mu_0 &\in \Mp{X_0} \\
       &\textrm{Peak}  &\mu_p &\in \Mp{[0, T] \times X} \\
        &\textrm{Occupation Start }  & \bmu_0 &\in \Mp{[0, T-\tau] \times X^2} \\        
        &\textrm{Occupation End }  & \bmu_1 &\in \Mp{[T-\tau, T] \times X^2} \\
        &\textrm{Time-Slack} & \nu &\in \Mp{[0, T] \times X} \label{eq:weak_solution_slack}
    \end{align}
\end{subequations}

The joint (relaxed) occupation measure $\bmu \in \Mp{[0, T] \times X^2}$ is constructed from the sum $\bmu = \bmu_0 + \bmu_1$. An \ac{MV} solution to the \ac{DDE} in \eqref{eq:delay_dynamics} is a set of measures from \eqref{eq:weak_solution} that satisfy three types of constraints: History-Validity, Liouville, Consistency.

\subsubsection{History-Validity}
\label{sec:history_validity}
The first History-Validity constraint is that $\mu_0$ should be a probability distribution over the initial state condition (at $t=0$). The second is that the history measure $\mu_h$ should represent an averaged occupation measure of histories that are defined between $[-\tau, 0]$, which implies that the $t$-marginal of $\mu_h$ should be Lebesgue-distributed. The two History-Validity constraints are,
\begin{align}
    \inp{1}{\mu_0} &= 1, &   \pi^{t}_\# \mu_h &= \lambda_{[-\tau, 0]}. \label{eq:History-Validity}
\end{align}

\subsubsection{Liouville}
The true occupation measure $(t,x_0,x_1) \mapsto \bmu(t, x_0, x_1)$ has a time $t$, a current state $x_0 = x(t \mid x_h)$, and an external input $x_1 \in X$ with $x_1(t) = x(t-\tau \mid x_h)$. 
Use of the Liouville equation in \eqref{eq:liou}  applied to the joint occupation measure $\bmu = \bmu_0 + \bmu_1$ leads to
\begin{align}
    & \mu_p = \delta_0 \otimes\mu_0 + \pi^{t x_0}_\# \Lie_f^\dagger (\bmu_0 + \bmu_1). \label{eq:peak_delay_liou}
\end{align}

\subsubsection{Consistency}
The $x_1$ input of $f$ from the Liouville equation \eqref{eq:peak_delay_liou} is not arbitrary; it should be equal to a time-delayed  $x_1(t) = x_0(t-\tau)$. This requirement will be imposed by a Consistency constraint. 

\begin{lem}
\label{lem:consistency_free}
    Let $x(\cdot)$ be a solution to \eqref{eq:delay_dynamics} for some history $x_h$ with an initial time of $0$ and a stopping time of $t^* \in [0, T]$. Then the following two integrals are equal for all  $\phi \in C([0, T] \times X)$:
    \begin{align}
    \label{eq:change_limits_free}
        &\left(\int_{0}^{t^*} + \int_{t^*}^{\min(T, t^*+\tau)}\right) \phi(t, x(t-\tau))dt  \nonumber\\
        &\qquad = \left(\int_{-\tau}^{0} + \int_{0}^{\min(t^*, T-\tau)} \right)\phi(t'+\tau, x(t)) dt'.
    \end{align}
\end{lem}
\begin{proof}
    This follows from 
a change of variable with $t' \leftarrow t - \tau$.
\end{proof}
Equation \eqref{eq:change_limits_free} inspires a consistency constraint for the free-terminal-time \ac{MV}-solution in \eqref{eq:weak_solution}. The left-hand-side of \eqref{eq:change_limits_free} may be generalized to 
\begin{equation}
\label{eq:consistency_lhs}
    \inp{\phi(t, x_1)}{\bmu_0(t, x_0, x_1) + \bmu_1(t, x_0, x_1)} + \inp{\phi(t, x)}{\nu(t, x)},
\end{equation}
in which $\bmu_0$ is supported in times $[0, \min(t^*, T-\tau)]$,  $\bmu_1$ is supported in times $[T-\tau, t^*]$ if $t^* > T-\tau$, and the slack measure $\nu$ implements the $[t^* , \min(T, t^*+\tau)]$ limits. The right-hand-side of \eqref{eq:change_limits_free} may be interpreted as \begin{equation}
\label{eq:consistency_rhs}
    \inp{\phi(t+\tau, x)}{\mu_h(t, x)} + \inp{\phi(t+\tau, x_0)}{\bmu_0(t, x_0, x_1)}.
\end{equation}

Define $S^\tau$  as the shift operator $S^\tau \phi(t, x) = \phi(t+\tau, x)$. With an abuse of notation, the pushforward operation $S^\tau_\#$ applied to a measure (such as $\mu_h$) will have the expression
\begin{align}
     \inp{\phi}{S^\tau_\#\mu_h} = \inp{S^\tau \phi}{\mu_h} =  \inp{\phi(t+\tau, x)}{\mu_h(t, x)}.
\end{align}
The Consistency constraint inspired by Lemma \ref{lem:consistency_free} is
\begin{equation}
\label{eq:consistency_free}
    \pi^{t x_1}_\# (\bmu_0 + \bmu_1) + \nu = S^\tau_\#(\mu_h + \pi^{t x_0}_\# \bmu_0). 
\end{equation}

\begin{rmk}
    Equation \eqref{eq:consistency_free} may also be written as $\pi^{t x_1}_\# (\bmu_0 + \bmu_1) \leq S^\tau_\#(\mu_h + \pi^{t x_0}_\# \bmu_0)$. The associated slack measure is $\nu$.
\end{rmk}



\subsection{Measure Program}

An infinite-dimensional \ac{LP} in terms of the measures from  \eqref{eq:weak_solution} to upper-bound Problem \eqref{eq:peak_delay_traj} is,
\begin{subequations}
\label{eq:peak_delay_meas}
    \begin{align}
        p^* = & \ \sup \quad \inp{p}{\mu_p} \label{eq:peak_delay_obj} \\
    & \inp{1}{\mu_0} = 1 \label{eq:peak_delay_prob}\\    
    & \pi^{t}_\# \mu_h = \lambda_{[-\tau, 0]} \label{eq:peak_delay_hist}\\   
    & \mu_p = \delta_0 \otimes\mu_0 + \pi^{t x_0}_\# \Lie_f^\dagger (\bmu_0 + \bmu_1) \label{eq:peak_delay_flow}\\        
    & \pi^{t x_1}_\# (\bmu_0 + \bmu_1) + \nu = S^\tau_\#(\mu_h + \pi^{t x_0}_\# \bmu_0) \label{eq:peak_delay_cons}\\ 
    & \textrm{Measure Definitions from  \eqref{eq:weak_solution}.} \label{eq:peak_delay_def}
    \end{align}
\end{subequations}
\begin{rmk}
    Membership in the history class $\hs$ is imposed  by the History-Validity constraint \eqref{eq:peak_delay_hist} and through support of $\mu_h$ in \eqref{eq:weak_solution_history}.
\end{rmk}

\begin{defn} 
\label{defn:mv_solution}
An \ac{MV}-solution to the \ac{DDE} \eqref{eq:delay_dynamics} with free-terminal-time and histories in $\hs$ is a tuple of measures that satisfy \eqref{eq:peak_delay_prob}-\eqref{eq:peak_delay_def} and \eqref{eq:weak_solution_history}-\eqref{eq:weak_solution_slack}.
\end{defn}

\begin{thm}
\label{thm:delay_upper_bound}
Under assumptions A1-A5, \eqref{eq:peak_delay_meas} will upper bound \eqref{eq:peak_delay_traj} with $p^* \geq P^*$ when $\hs$ is graph-constrained.
\end{thm}
\begin{proof}
This proof will proceed by demonstrating that every $(t^*, x_h)$ candidate from \eqref{eq:peak_delay_traj} may be expressed by a unique \ac{MV}-solution from Defn. \ref{defn:mv_solution}. The history measure $\mu_h$ is the $[-\tau, 0]$ occupation measure of $x(t)$, and the initial measure $\mu_0$ is the Dirac-delta $\delta_{x_h(0^+)}$. The peak measure $\mu_p$ is the Dirac-delta $\delta_{t=t^*} \otimes \delta_{x = x(t^* \mid x_h)}$. 
The relaxed occupation measures $(\bmu_0, \bmu_1, \nu)$ will now be considered. For convenience, define $z(t) = (t, x(t \mid x_h), x(t-\tau \mid x_h))$ as the delay embedding of the trajectory $x(t \mid x_h)$. In the case where $t^* \in [0, T-\tau]$, then $\bmu_0$ is the $[0, t^*]$ occupation measure of $z(t)$, $\bmu_1$ is the zero measure, and $\nu$ is the $[t^*, t^*+\tau]$ occupation measure of $(t, x(t-\tau \mid x_h))$. Alternatively when $t^* \in (T-\tau, T]$, $\bmu_0$ is the $[0, T-\tau]$ occupation measure of  $z(t)$, $\bmu_1$ is the $[T-\tau, t^*]$ occupation measure of $z(t)$, and $\nu$ is the $[t^*, T]$ occupation measure of $(t, x(t-\tau \mid x_h))$.
All of the measures in \eqref{eq:weak_solution} have been defined for each input $(t^*, x_h)$, which proves that $p^*\geq P^*$. 
\end{proof}


Appendix \ref{app:delay_structure} uses these methods to form \ac{MV}-solutions to systems with other delay structures (proportional delay, long-delay discrete-time systems). 

\begin{rmk}
    The proof of Theorem  \ref{thm:delay_upper_bound} provides a unique \ac{MV} solution for each \ac{DDE} trajectory. Additionally, each \ac{DDE} trajectory given an initial condition $x_h$ is unique under the Lipschitz assumption A2.

    We note that \ac{MV} solutions are not necessarily unique (for a given terminal time distribution $\pi^t_\# \mu_p$) when the history measure $\mu_h$ is supported on the graph of more than one curve. As an example, Figure \ref{fig:same_occ} shows two sets of curves under the dynamics $\dot{x}(t) = -2x(t) - 3x(t-1)$ in the times $t\in [0, 5]$. The history occupation measure $\mu_h = 0.5 \lambda_{[-1, 0]} \otimes \delta_{x=1} + 0.5 \lambda_{[-1, 0]} \otimes \delta_{x=-1}$ is supported in the set $\mu_h \in \Mp{[-1, 0] \times \{-1, 1\}}$. The superposition of each set of red and blue curves each have the same history measure $\mu_h$, but the switch that takes place on the bottom plot (e.g. blue: $x_h(t) = 1$ for $t \in [-1, -0.5)$, \ $x_h(t) = -1$ for $t \in [0.5, 0]$) yields a different trajectory going forward in time.
    

\begin{figure}[h]
    \centering
    \includegraphics[width=0.7\linewidth]{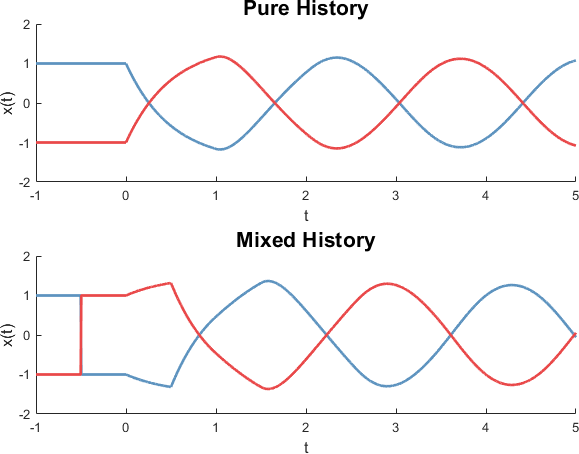}
    \caption{The same $\mu_h$ leads to different trajectories in times $(0, 5]$}
    \label{fig:same_occ}
\end{figure}
    
\end{rmk}

\subsection{Function Program}

The dual program of \eqref{eq:peak_delay_meas} with variables $(\gamma, \xi, v, \phi)$ is
\begin{subequations}
\label{eq:peak_delay_cont}
    \begin{align}
        d^* = & \inf_{\gamma \in \R} \ \gamma + \int_{-\tau}^0 \xi(t) dt \label{eq:peak_delay_cont_obj} \\
        & \xi(t) + \phi(t+\tau, x) \geq 0 & & \forall (t, x) \in H_0 \\ 
        & \gamma \geq v(0, x)  & & \forall x \in X_0 \\ 
        & v(t, x) \geq  p(x)  & & \forall (t, x) \in [0, T] \times X\\ 
        & \phi(t, x) \leq 0 & & \forall (t, x) \in [0, T] \times X\\ 
        & \Lie_f v(t, x_0) + \phi(t, x_1) \leq \phi(t+\tau, x_0) & & \forall (t, x_0, x_1) \in [0, T-\tau] \times X^2 \label{eq:peak_delay_cont_lie0}\\ 
        & \Lie_f v(t, x_0) + \phi(t, x_1) \leq 0 & & \forall (t, x_0, x_1) \in [T-\tau, T] \times X^2 \label{eq:peak_delay_cont_lie1}\\
        & \xi \in C([-\tau, 0]) \\
        & v \in C^1([0, T] \times X) \\
        & \phi \in C([0, T] \times X).
    \end{align}
\end{subequations}

\begin{thm}
    There is no duality gap between \eqref{eq:peak_delay_meas} and \eqref{eq:peak_delay_cont}.
\end{thm}
\begin{proof}
    See Appendix \ref{app:strong_duality_delay}.
\end{proof}

We pose the following conjecture based on \cite{lewis1980relaxation, rosenblueth1991relaxation}:
\begin{conj}
\label{conj:delay}
Assume that A1-A5 hold. Additionally, assume that $T > \tau>0$ and the image-set $f(t, x_0, X)$ is convex for all fixed $t \in [0, T], \ x_0 \in X$. Then there is no relaxation gap between \eqref{eq:peak_delay_traj} and \eqref{eq:peak_delay_meas} ($p^*=P^*$).
\end{conj}

Proving Conjecture \ref{conj:delay} is the subject of ongoing work.

Appendix \ref{app:ocp} contains a further discussion of the continuity and structural aspects of the dual solution in \eqref{eq:peak_delay_cont} as applied to bounding costs on \ac{DDE} \acp{OCP}.


\section{Peak Moment Program}
\label{sec:moment}
This section will briefly review the moment-\ac{SOS} hierarchy \cite{lasserre2009moments} in order to approximate-from-above Program \eqref{eq:peak_delay_meas} by a sequence of finite-dimensional \acp{SDP}.
\subsection{Review of Moment-SOS Hierarchy}

Let $\mu\in \Mp{X}$ be a measure, and let $\alpha \in \N^n$ be a multi-index. The $\alpha$-moment of $\mu$ is the pairing $\bm_\alpha = \inp{x^\alpha}{\mu}$. The moment sequence $\bm = \{\bm_\alpha\}_{\alpha \in \N^n}$ is the infinite collection of moments of $\mu$. A unique (Riesz) linear functional $L_\bm$ exists operating on each polynomials $p =\sum_{\alpha \in \mathcal{J}} p_\alpha x^\alpha \in \R[x]$  as $L_\bm(p) = \sum_{\alpha \in \mathcal{J}} p_\alpha \bm_\alpha$ for a finite index set $\mathcal{J} \subset \N^n$.

A set is \ac{BSA} if it is defined by a finite number of polynomial inequality constraints, such as by $\K = \{x \in \R^n \mid g_k(x) \geq 0: \ k = 1 \ldots N_c\} \subseteq \R^n$. The measure $\mu$ is supported on $\K$ if $\mu \in \Mp{\K}$.
Given a polynomial $g = \sum_{\gamma \in \mathcal{J}} g_\gamma x^\gamma$, the localizing matrix $\M[g \bm]$ induced by the  constraint $g(x) \geq 0$ with respect to the moment sequence $\bm$ is the infinite-dimensional matrix indexed by $\alpha, \beta \in \N^n$ as $\M[g \bm]_{\alpha, \beta} = L_{\bm}(x^{\alpha+\beta} g) = \sum_{\gamma \in \R^n} g_{\gamma} \bm_{\alpha+\beta+\gamma}$. The moment matrix $\M[ \bm]$ is the localizing matrix associated with $g = 1$. The matrix $\M[\K \bm]$ is the block-diagonal matrix comprised of $\M[\bm]$ and $\M[g_k \bm]$ for $k=1 \ldots N_c$.

Let $\{\tilde{\bm}_{\alpha}\}_{\alpha \in \N^n}$ be a sequence of real numbers. If there exists some measure $\tilde{\mu} \in \Mp{\K}$ such that $\forall \alpha \in \N^n: \ \inp{x^\alpha}{\mu} = \tilde{\bm}_\alpha$ then $\tilde{\mu}$ is a representing measure for $\tilde{\bm}$, and $\tilde{\bm}$ is a moment-sequence for $\tilde{\mu}$. Such a representing measure (if it exists) could be nonunique. The stronger condition that there is a unique representing measure for $\tilde{\bm}$ is called moment determinacy. 
A necessary condition for $\tilde{\bm}$ to have a representing measure is that the block-diagonal matrix $\M[\tilde \bm]$ is \ac{PSD}. This necessary condition is also sufficient if $\K$ satisfies an \textit{Archimedean} condition (stronger than compactness, equivalent after a ball constraint $R - \norm{x}_2^2 \geq 0$ is added to $\K$ for sufficiently large $R>0$ if $\K$ is compact). 
In general we will call $\tilde \bm$ a \textit{pseudo-moment} sequence.

The order-$d$ truncation of $\M[\K \bm]$ (for $d \in \N$ and expressed as $\M_d[\K \bm]$) keeps entries of degree $\leq 2d$, and preserves the top-corner of each matrix in the block-diagonal. 
The moment matrix $\M_d[\bm]$ is a \ac{PSD} matrix of size $\binom{n+d}{d}$ assuming a monomial basis for $x$ is employed. 
The size of each truncated localizing matrix $\M_d[g_k \bm]$ is $\binom{n+d-\ceil{d_k/2}}{d-\ceil{d_k/2}}$, where $d_k = \deg g_k$. 
The moment-\ac{SOS} hierarchy is the process of increasing the degree $d \rightarrow \infty$ when forming moment programs associated to measure \acp{LP}.

\subsection{Moment Program}

Additional assumptions are required in order to approximate \eqref{eq:peak_delay_meas} using the moment-\ac{SOS} hierarchy:
\begin{itemize}
    \item[A6] The sets $H_0$, $X_0$, and $X$ are Archimedean \ac{BSA} sets.
    \item[A7] Both $p$ and $f$ are polynomials.
\end{itemize}

Let the measures $(\mu_h, \mu_0, \mu_p, \bmu_0, \bmu_1, \nu)$ have associated pseudo-moment sequences $(\bm^h, \bm^0, \bm^p, \bar{\bm}^0, \bar{\bm}^1, \bm^\nu)$ respectively. Let $\alpha \in \N^n$ and $\beta \in \N$ be multi-indices that define monomial test functions $x_0^\alpha t^\beta$. For each multi-index 
tuple $(\alpha, \beta)$, the operator $\textrm{Liou}_{\alpha \beta}(\bm^0, \bm^p, \bar{\bm}^0, \bar{\bm}^1)$  may be derived from the linear relations induced by the Liouville equation \eqref{eq:peak_delay_flow} (in which $\delta_{\beta0}=1$ is a Kronecker delta):  
 \begin{align}
 0 &= \inp{x^\alpha}{\mu_0} \delta_{\beta 0} + \inp{\Lie( x_0^\alpha t^\beta)}{\bmu^0 + \bmu^1} - \inp{x^\alpha t^\beta}{\mu_\tau}.
     \end{align}

Similarly, the operator     $\textrm{Cons}_{\alpha \beta}(\bm^h, \bm^\nu, \bar{\bm}^0, \bar{\bm}^1,)$ may be derived from the consistency constraint \eqref{eq:peak_delay_cons} by
\begin{align}
    0 = & \inp{x_1^\alpha t^\beta}{\bmu^0 + \bmu^1} + \inp{x^\alpha t^\beta}{\nu} - \inp{x^\alpha(t+\tau)^\beta}{\mu_h} \\
    &  -\inp{x_0^\alpha(t+\tau)^\beta}{\bmu^0}. \nonumber
\end{align}

Given a degree $d \in \N$, the dynamics degree $\tilde{d} \geq d$  may be defined as $\tilde{d}=d + \floor{\deg f /2}.$

\begin{prob}
Program \eqref{eq:peak_delay_meas} is upper-bounded by the following order-$d$ \ac{LMI} in pseudo-moments:
\begin{subequations}
\label{eq:peak_delay_lmi}
\begin{align}
    p^*_{d} = & \max \quad  L_{\bm^p}(p)  \label{eq:peak_delay_lmi_obj} \\
        & \bm^0_0 = 1 \label{eq:peak_delay_lmi_prob}\\
    &  \forall (\alpha, \beta) \in \N^{n+1}_{\leq 2d}: \nonumber \\
        & \qquad \bm^h_{\beta} = \textstyle \int_{-\tau}^0 t^\beta dt = -(-\tau)^{\beta+1}/(\beta+1) \label{eq:peak_delay_lmi_leb}  \\
&\qquad \textrm{Liou}_{\alpha \beta}(\bm^0, \bm^p, \bar{\bm}^0, \bar{\bm}^1) = 0  \label{eq:peak_delay_lmi_flow} \\
    &\qquad \textrm{Cons}_{\alpha \beta}(\bm^h, \bm^\nu, \bar{\bm}^0, \bar{\bm}^1,)= 0 \label{eq:peak_delay_lmi_cons}   \\
    & \M_d((X_0) \bm^0), \ \M_{\tilde{d}}((H_0)\bm^h) \succeq 0   \label{eq:peak_delay_lmi_supp_beg} \\
    &\M_d(([0, T] \times X)\bm^p) \succeq 0 \\
    &\M_{\Tilde{d}}(([0, T-\tau] \times X^2)\bar{\bm}^0) \succeq 0 \\
    &\M_{\Tilde{d}}(([T-\tau, T] \times X^2)\bar{\bm}^1) \succeq 0 \\
    &\M_{\tilde{d}}(([0, T] \times X)\bm^\nu) \succeq 0. \label{eq:peak_delay_lmi_supp_end} 
\end{align}
\end{subequations}
\end{prob}
The objective \eqref{eq:peak_delay_lmi_obj} is the pseudo-moment version of $\inp{p}{\mu_p}$. 
Constraints \eqref{eq:peak_delay_lmi_leb} and \eqref{eq:peak_delay_lmi_prob} are History-Validity constraints from \eqref{eq:History-Validity} when applied to the pseudo-moments $(\bm^\nu, \bm^0)$. Constraints \eqref{eq:peak_delay_lmi_flow} and \eqref{eq:peak_delay_lmi_cons} are the Liouville and Consistency constraints respectively. Constraints \eqref{eq:peak_delay_lmi_supp_beg}-\eqref{eq:peak_delay_lmi_supp_end} are  support constraints necessary for the pseudo-moments to have representing measures.

Boundedness of all moments of measures in \eqref{eq:weak_solution} is required to obtain convergence of \eqref{eq:peak_delay_lmi} to \eqref{eq:peak_delay_meas} as $d \rightarrow \infty$.
\begin{lem}
\label{lem:moment_bound}
All measures from \eqref{eq:weak_solution}  in an \ac{MV}-solution (Defn. \ref{defn:mv_solution}) are bounded under assumptions A1-A7.
\end{lem}

\begin{proof}
Boundedness of a measure's mass and support is a  sufficient condition that all of the measure's moments are bounded. Assumption A1 ensures compactness, with the requirement from Defn. \ref{defn:graph_constrained} that $H_0 \subseteq [-\tau, X]$ and $X_0 \subseteq X$. The remainder of this proof will involve finding upper bounds on the masses of all measures in \eqref{eq:weak_solution}.

The initial measure $\mu_0$ has a mass of 1, and the history measure $\mu_h$ has a mass of $\tau$ by the History-Validity constraints \eqref{eq:peak_delay_prob} and \eqref{eq:peak_delay_hist}.
Substitution of the test function $v(t, x) = 1$ in the Liouville \eqref{eq:peak_delay_flow} leads to $\inp{1}{\mu_p} = \inp{1}{\mu_0} = 1$. Since $T$ is finite, the moment $\inp{t}{\mu_p} \leq \inp{1}{\mu_p} \; (\sup_{t \in [0, T]} t) = T$ is also finite. Use of the test function $v(t, x) = t$ into the Liouville \eqref{eq:peak_delay_flow}  yields $\inp{t}{\mu_p} = \inp{1}{\bmu_0 +\bmu_1} \leq T$. 
Because $\bmu_0$ and $\bmu_1$ are both nonnegative Borel measures, it holds that $\inp{1}{\bmu_0} \leq T$ and $\inp{1}{\bmu_1} \leq T$. The final constraint involves substitution of $\phi(t, x) = 1$ into the Consistency \eqref{eq:peak_delay_cons}, resulting in
\begin{align}
\label{eq:consistency_substitute}
\inp{1}{\bmu_0 + \bmu_1} + \inp{1}{\nu} &= \inp{1}{\mu_h} + \inp{1}{\bmu_0} \\
\inp{1}{\nu} &= \inp{1}{\mu_h} - \inp{1}{\bmu_1} = \tau - \inp{1}{\bmu_1}. \nonumber
\end{align}
Given that $\bmu_1$ and $\nu$ are nonnegative Borel measures and cannot have negative masses, the mass $\inp{1}{\nu}$ is constrained within $[0, \tau]$. 
All masses are demonstrated to be finite, thus proving boundedness.
\end{proof}

\begin{rmk}
Neglecting the History-Validity constraint \eqref{eq:peak_delay_hist} allows for $\mu_h$ in \eqref{eq:consistency_substitute} to have infinite mass, violating the boundedness principle.
\end{rmk}

\begin{thm} The optima in \eqref{eq:peak_delay_lmi} will converge as $\lim_{d \rightarrow \infty} p^*_d = p^*$ to \eqref{eq:peak_delay_meas} under assumptions A1-A6.\end{thm}
\begin{proof}
    This follows from Corollary 8 of \cite{tacchi2022convergence} under the boundedness condition in Lemma \ref{lem:moment_bound}.
\end{proof}

\begin{rmk}
    Assumption $A6$ can be generalized to cases where the sets ($H_0$, $X_0$, $X$) are the unions of \ac{BSA} sets. As an example, consider $H_0 = H_0^1 \cup H_0^2$ in which $\pi^t H_0^1 = [-\tau, -\tilde{\tau}]$ and $\pi^t H_0^2= [-\tilde{\tau}, 0]$ for some $\tilde{\tau} \in (0, \tau)$. Then the pseudo-moments $\bm^h = \bm^h_1 + \bm^h_2$ can be implicitly constructed from $\M_d((H_0^1)\bm_1^h), \;\M_d((H_0^2)\bm_2^h) \succeq 0$.
\end{rmk}

\subsection{Computational Complexity}


Table \ref{tab:moment_size} lists the size of the  order-$d$ \ac{PSD} moment matrices associated with the pseudo-moment sequences $(\bm^h, \bm^0, \bm^p, \bar{\bm}^0, \bar{\bm}^1, \bm^\nu)$.

\begin{table}[h]
    \centering
    \caption{Size of Moment Matrices in \ac{LMI} \eqref{eq:peak_delay_lmi}}
    \begin{tabular}{rccc}
         Matrix: &$\M_d{(\bm^0)}$ & $\M_{\tilde{d}}{(\bm^p)}$ & $\M_d{(\bm^h)}$  \\
         Size: & $\binom{n+d}{d}$ &$\binom{n+1+d}{d}$ & $\binom{n+1+\tilde{d}}{\tilde{d}}$ \\ \\
         Matrix: &$\M_d{(\bar{\bm}^0)}$ & $\M_{\tilde{d}}{(\bar{\bm}^1)}$ & $\M_d{(\bm^{\nu})}$  \\
         Size: & $\binom{2n+1+\tilde{d}}{\tilde{d}}$ &$\binom{2n+1+\tilde{d}}{\tilde{d}}$ & $\binom{n+1+\tilde{d}}{\tilde{d}}$ \\
    \end{tabular}
    \label{tab:moment_size}
\end{table}
The largest size written in Table \ref{tab:moment_size} is $\binom{2n+1+\tilde{d}}{\tilde{d}}$, which occurs with the pseudo-moment sequences $(\bar{\bm}^0, \bar{\bm}^1)$ associated to the two joint occupation measures $(\bmu_0, \bmu_1)$. Equality constraints between entries of the moment matrices must be added to convert the \ac{LMI} into \iac{SDP} for use in symmetric-cone Interior Point Methods. The per-iteration complexity of solving an \ac{SDP}  derived from an order-$d$ \ac{LMI} involved in the moment-\ac{SOS} hierarchy  scales as $O(n^{6d})$ \cite{lasserre2009moments} with $n$. In the case of \ac{LMI} \eqref{eq:peak_delay_lmi}, the computational complexity of solving \eqref{eq:peak_delay_lmi} will scale approximately as $(2n+1)^{6 \tilde{d}}$  (based on $\bar{\bm}^0, \bar{\bm}^1$).

\subsection{Distance Estimation}

The distance estimation framework of \cite{miller2022distance_short} may also be applied towards \acp{DDE}. The \ac{DDE} distance estimation program with an unsafe set $X_u \subset X$, a metric $c(x, y)$, and point-unsafe-set distance function $c(x; X_u) = \inf_{y \in X_u} c(x, y)$ is
\begin{subequations}
\label{eq:dist_delay_traj}
    \begin{align}
    P^* = & \inf_{t^* \in [0, T], \; x_h(\cdot)} c(x(t^* \mid x_h); X_u) &\\
    & \dot{x} =  f(t, x(t), x(t-\tau)) & & \forall t \in [0, T]  \\
    & x(t) = x_h(t) & & \forall t \in [-\tau, 0]\\
     & x_h(\cdot) \in \mathcal{H}.
    \end{align}
\end{subequations}

Safety in program \eqref{eq:dist_delay_traj} is measured pointwise: a trajectory is safe if $x(t \mid x_h) \not \in X_u$ for every time $t \in [0, T]$. Safety ensuring that the entire history is never contained within $X_u$ ($\exists s \in [-\tau, 0] \mid x(t-s \mid x_h) \not\in X_u \forall t \in [0, T]$) is a more challenging separate problem and will not be considered here.

The \ac{MV}-solution in \eqref{eq:weak_solution} may be applied to \eqref{eq:dist_delay_traj} to create \iac{DDE} distance estimation program by adding a joint probability measure $\eta \in \Mp{X \times X_u}$ (following the procedure from \cite{miller2022distance_short}):
\begin{subequations}
\label{eq:dist_delay_meas}
    \begin{align}
        c^* = & \ \inf \quad \inp{c}{\eta} \label{eq:dist_delay_obj} \\
        & \pi^x_\# \mu^p = \pi^x_\# \eta \label{eq:dist_delay_marg}\\
    & \inp{1}{\mu_0} = 1 \label{eq:dist_delay_prob}\\    
    & \pi^{t}_\# \mu_h = \lambda_{[-\tau, 0]} \label{eq:dist_delay_hist}\\   
    & \mu_p = \delta_0 \otimes\mu_0 + \pi^{t x_0}_\# \Lie_f^\dagger (\bmu_0 + \bmu_1) \label{eq:dist_delay_flow}\\        
    & \pi^{t x_1}_\# (\bmu_0 + \bmu_1) + \nu = S^\tau_\#(\mu_h + \pi^{t x_0}_\# \bmu_0) \label{eq:dist_delay_cons}\\ 
    & \eta \in \Mp{X \times X_u} \\
    & \textrm{Measure Definitions from  \eqref{eq:weak_solution}.} \label{eq:dist_delay_def}
    \end{align}
\end{subequations}

The distance estimation program only affects the cost \eqref{eq:dist_delay_obj}. This change is orthogonal to the modification in dynamics necessary to create \iac{DDE} \ac{MV}-solution from the \ac{ODE} program.

\section{Numerical Examples}
\label{sec:delay_examples}
All experiments were developed in MATLAB 2021a, and code is available at \url{https://github.com/Jarmill/timedelay}. Dependencies include Gloptipoly \cite{henrion2003gloptipoly}, YALMIP \cite{Lofberg2004}, and Mosek \cite{mosek92} in order to formulate and solve moment-\ac{SOS} \acp{LMI} and \acp{SDP}. 

In this section, a notational convention where $(x_1, x_2)$ correspond to coordinates of $x \in X$ will be used. All sampled histories in visualizations are piecewise-constant inside $H_0$ with 10 randomly-spaced jumps between $[-\tau, 0]$.




\label{sec:delay_delay_traj_analysis_peak_example}


\subsection{Epidemic Model}
This section provides an example of \iac{MV}-solution and peak estimation given a single history in a compartmental epidemic model.
Many diseases have incubation periods during which there is a delay between initial infection and infectious potential. In the current COVID-19 pandemic, this incubation period appears to be between 2-14 days, with a median of 5 days \cite{lauer2020incubation}.  The epidemic  dynamics with time delays are
\begin{subequations}
\label{eq:sir_delay}
\begin{align}
    S'(t) &= -\beta S(t) I(t) \\
    I'(t) &= \beta S(t-\tau) I(t-\tau) - \gamma I(t) \\
    R'(t) &= \gamma I(t)
\end{align}
\end{subequations}
There exists also exists a `latent' state $L'(t) = \beta S(t) I(t) - \beta  S(t - \tau) I (t - \tau)$ such that $S + I + R + L = 1$. The setting discussed in this section is $\beta = 0.4, \ \gamma = 0.1, \  \ T = 30$. 

Figures \ref{fig:sir_01} and \ref{fig:sir_02} display simulations of this epidemic model as $\tau$ changes under a constant state history with $R=0$. The black curve in Figures \ref{fig:sir_01} and \ref{fig:sir_02} is the plot of $I(t)$ at $\tau=0$. As the incubation period $\tau$, the time $t^*$ at which the peak is achieved is delayed (moves rightwards) in a monotonically increasing manner. The other colored curves in each plot have delays $\tau \in 1..9$.
In Figure \ref{fig:sir_01} with $I_h = 0.1$, the peak infected population decreases as the delay $\tau$ increases. Conversely in Figure \ref{fig:sir_02} with $I_h = 0.2$, the peak infected population increases as the delay increases.

\begin{figure}[ht]
     \centering
     \begin{subfigure}[b]{0.48\linewidth}
         \centering
         \includegraphics[width=\linewidth]{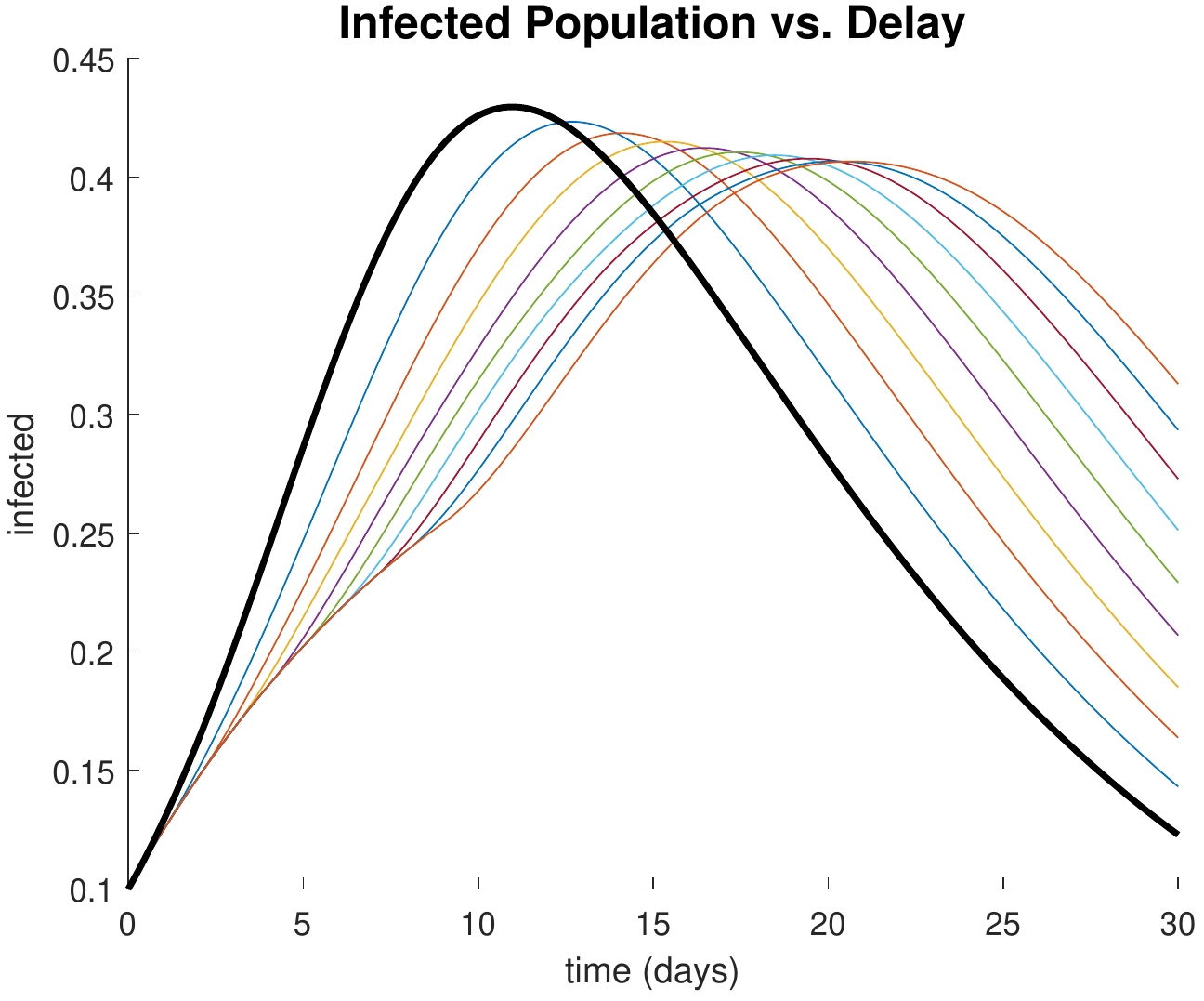}
         \caption{\label{fig:sir_01} $I_h = 0.1$, peak decreases}
         
     \end{subfigure}
     \;
     \begin{subfigure}[b]{0.48\linewidth}
         \centering
         \includegraphics[width=\linewidth]{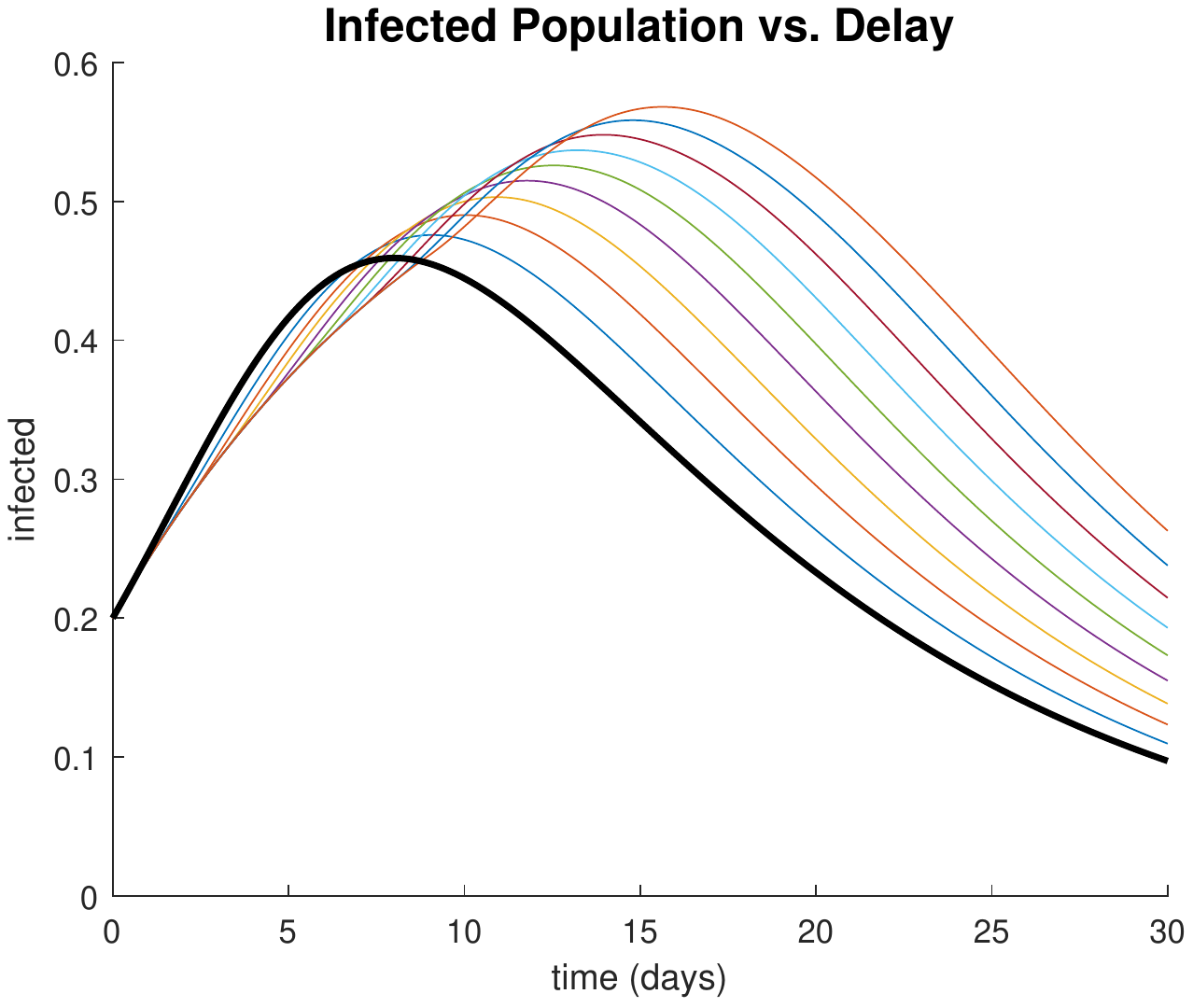}
         \caption{\label{fig:sir_02} $I_h = 0.2$, peak increases}
     \end{subfigure}
      \caption{\label{fig:sir_peak} Peak infected population vs. time delay}
\end{figure}

For the peak estimation example, a  constant history is assumed with  an initial infection rate of $I_h = 0.2$ and an incubation period of $\tau = 9$, forming the initial history $S(t) = 1-I_h, I(t) = I_h, R(t) = 0$ between $t \in [-9, 0]$.

For numerical purposes the dynamics are scaled such that $\tilde{t} \in [0, 1]$ with an effective delay of $\tilde{\tau} = \tau / T = 0.3$.
Only the $x = (S, I)$ subsystem is considered to form the state set $X = \{S \geq 0, I \geq 0, S+I \leq 1\}$, and the joint occupation measures $(\bar{\mu}_0, \bar{\mu}_1)$ have variables  $(t, S_0, I_0, S_1, I_1)$. 



Peak estimation is employed to bound the maximum infection rate over the course of the epidemic. This peak estimation program maximizes $\inp{I}{\mu_p}$ under the constraint that $(\mu_p, \bar{\mu}, \{\nu_0, \nu_1\}, \hat{\nu}_1)$ is a free-time \ac{MV}-solution of dynamics \eqref{eq:sir_delay}.

Figure \ref{fig:sir_peak_recovery} displays the output of peak estimation, where the order-3 LMI relaxation bounds the maximal infection rate at 56.89\%. The moment matrix $\M_3[y_p]$ is approximately rank-1 (second largest eigenvalue of $\M_3[y_p] = 2.448\times 10^{-5}$), and the extracted optimum from $\M_3[y_p]$ by Algorithm 1 of \cite{miller2020recovery} is $(S^*, I^*) = (0.0561, 0.5689)$ occurring at $t^*=15.636$ days.
\begin{figure}[h]
    \centering
    \includegraphics[width=0.6\linewidth]{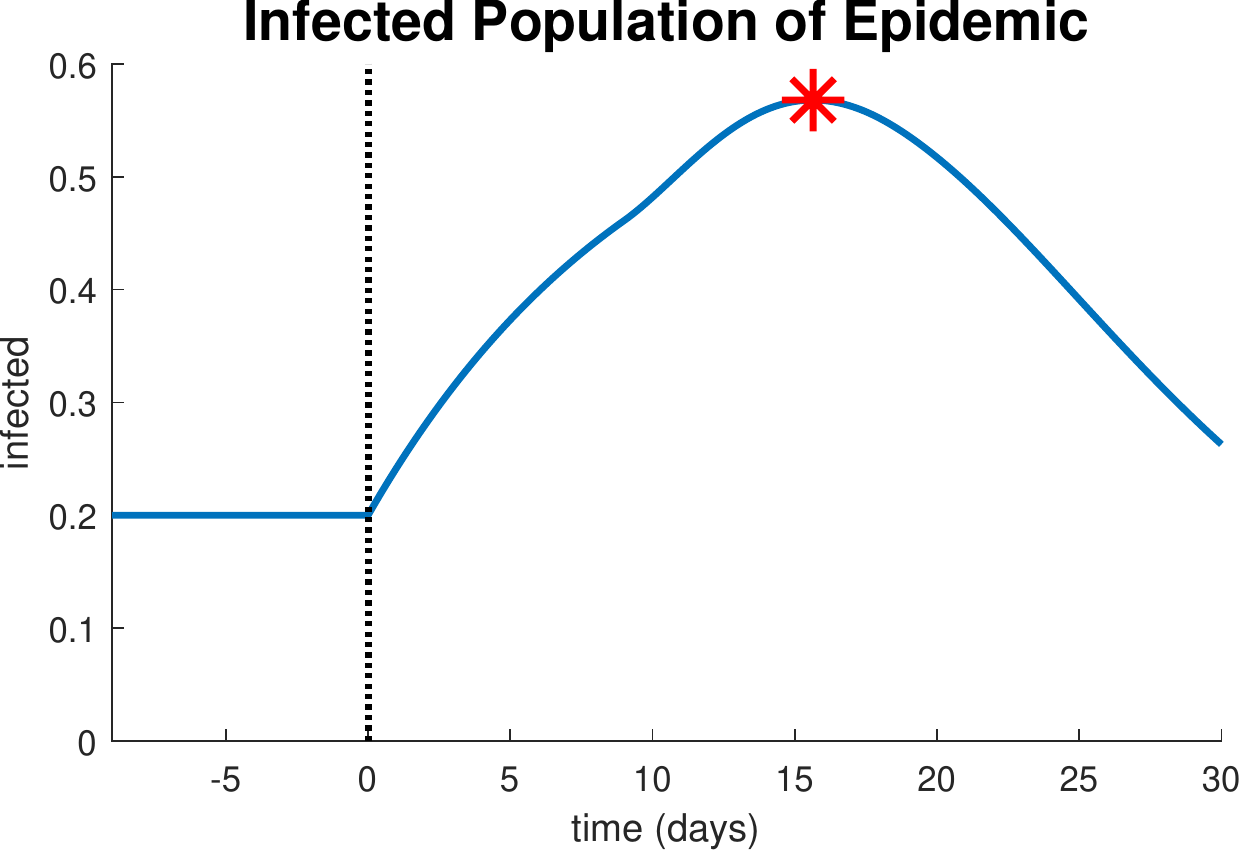}
    \caption{\label{fig:sir_peak_recovery} SIR peak estimation and recovery at order 3}
    
\end{figure}
\subsection{Delayed Flow System}
\label{ex:delay_flow}
A time-delayed version of the Flow system from \cite{prajna2004safety} is
\begin{equation}
\label{eq:flow_delay}
    \dot{x}(t) = \begin{bmatrix}x_2(t) \\ -x_1(t-\tau) - x_2(t) + x_1(t)^3/3
        \end{bmatrix}.
\end{equation}

Figure \ref{fig:flow_delay_comparision} plots the delayed Flow system \eqref{eq:flow_delay} without lag ($\tau=0$ in blue) and with a lag ($\tau=0.75$ in orange) starting from the constant initial history $x_h(t) = (1.5, 0), \ \forall t \in [-\tau, 0]$ (black circle). 
    \begin{figure}[ht]
        \centering
        \includegraphics[width=0.5\linewidth]{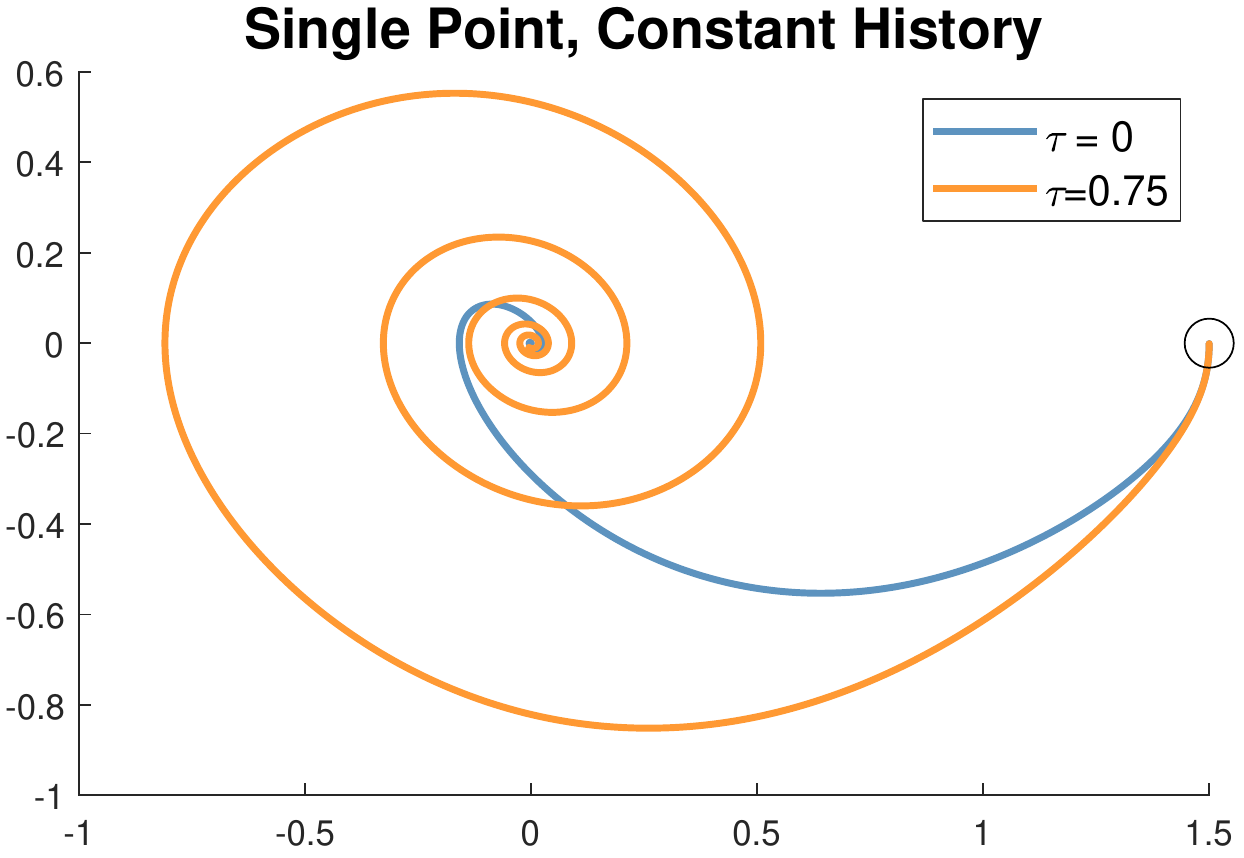}
        \caption{\label{fig:flow_delay_comparision}Comparison of delayed Flow systems \eqref{eq:flow_delay} with lags $\tau=0$ and $\tau= 0.75$ in times $t \in [0, 20]$ }
    \end{figure}

The time-zero set of allowable histories is $X_0 = \{x \in \R^2 \mid  \ (x_1-1.5)^2 + x_2^2 \leq 0.4^2\}$. The history class $\hs$ will be the set of functions $x_h \in PC([-\tau, 0])$ whose graphs $(t, x(t))$ are contained within  the cylinder $H_0 = [-0.75, 0] \times X_0$. No further requirements of continuity are posed on histories in $\hs$.     
The considered peak estimation aims to find the minimum value of $x_2$ (maximize $p(x)=-x_2)$
for trajectories following \eqref{eq:flow_delay} starting from $H_0$, within the state set 
$X = [-1.25, 2.5] \times [-1.25, 1.5]$ and time horizon $T=5$. The first five bounds on the maximum value of $-x_2$ by solving \eqref{eq:peak_delay_lmi} are $p^*_{1:5}=[1.25, 1.2183, 1.1913, 1.1727, 1.1630]$.

Figure \ref{fig:flow_delay_peak} plots trajectories and peak information associated with this example. The black circle is the initial set $X_0$. The initial histories inside $X_0$ are plotted in grey. These sampled histories are piecewise constant with 10 uniformly spaced jumps (moving to a new point uniformly sampled in $X_0$) within $[-0.75, 0]$. The cyan curves are the \ac{DDE} trajectories of \eqref{eq:flow_delay} starting from the grey histories. The red dotted line is the $p^*_5$ bound on the minimum vertical coordinate of a point on any trajectory starting from $\hs$ up to $T=5$.
    \begin{figure}[ht]
        \centering
        \includegraphics[width=0.5\linewidth]{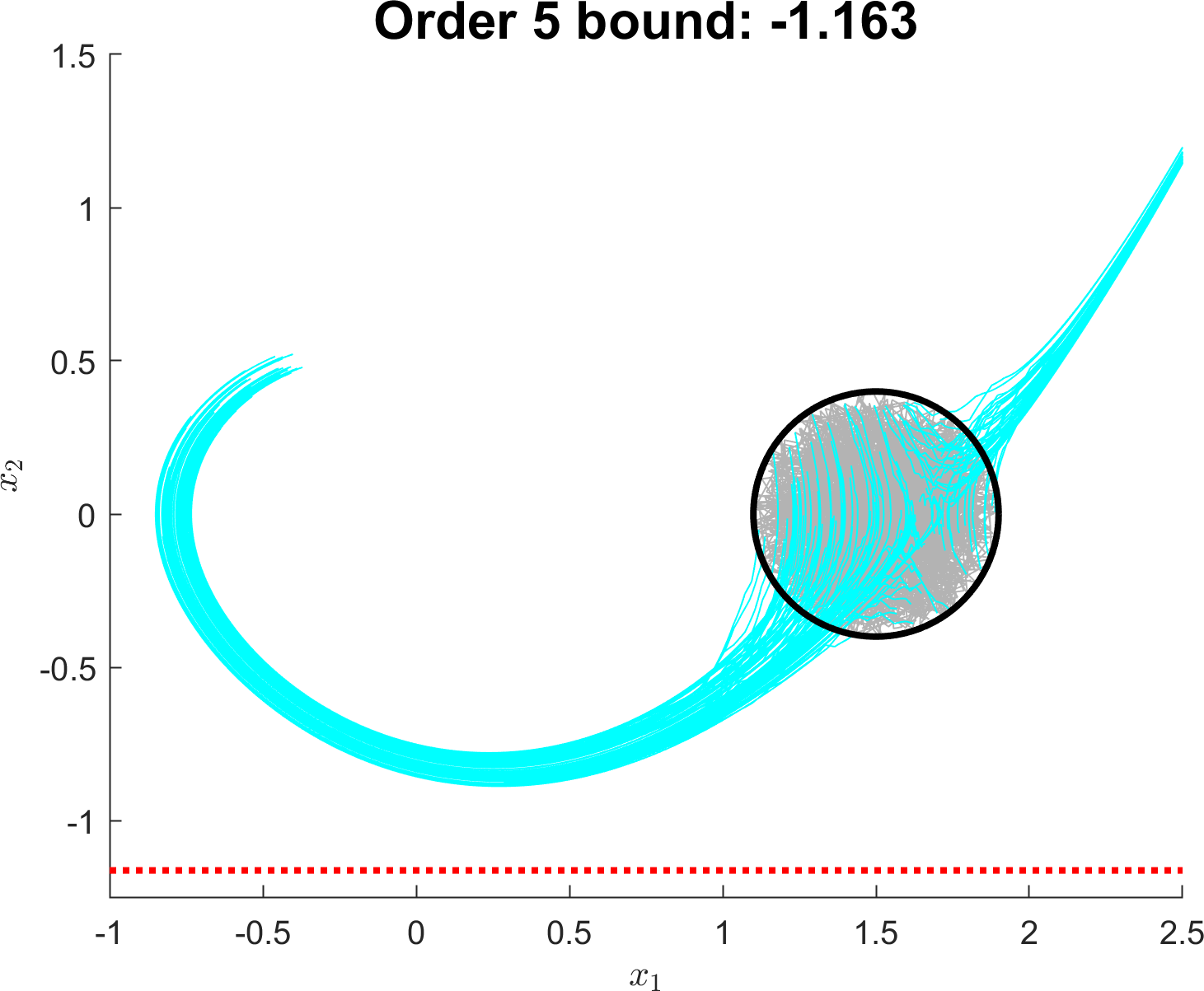}
        \caption{\label{fig:flow_delay_peak} Minimize $x_2$ on the delayed Flow system \eqref{eq:flow_delay}}
    \end{figure}

Distance estimation is performed on the Flow system \ref{eq:flow_delay} with an $L_2$ metric, a time horizon of $T=8$, arbitrarily varying histories in $H_0$, a time horizon of $\tau=0.5$, and a half-circle unsafe set  $X_u = \{x \mid 0.5^2 \geq (x_1+0.5)^2 + (x_2+1)^2, (1.5+x_1+x_2)$. The recovered distance estimates up to degree 4 from \ac{SDP} relaxations of \eqref{eq:dist_delay_meas} are $c^*_{1:4} = [1.1897 \times 10^{-4}, \  4.0420 \times 10^{-4}, \ 0.1572, \ 0.1820]$. 
Figure \ref{fig:flow_delay_dist} plots the set $X_u$ in red along with its $c^*_4 = 0.1820$ certified distance contour.
    \begin{figure}[ht]
        \centering
        \includegraphics[width=0.5\linewidth]{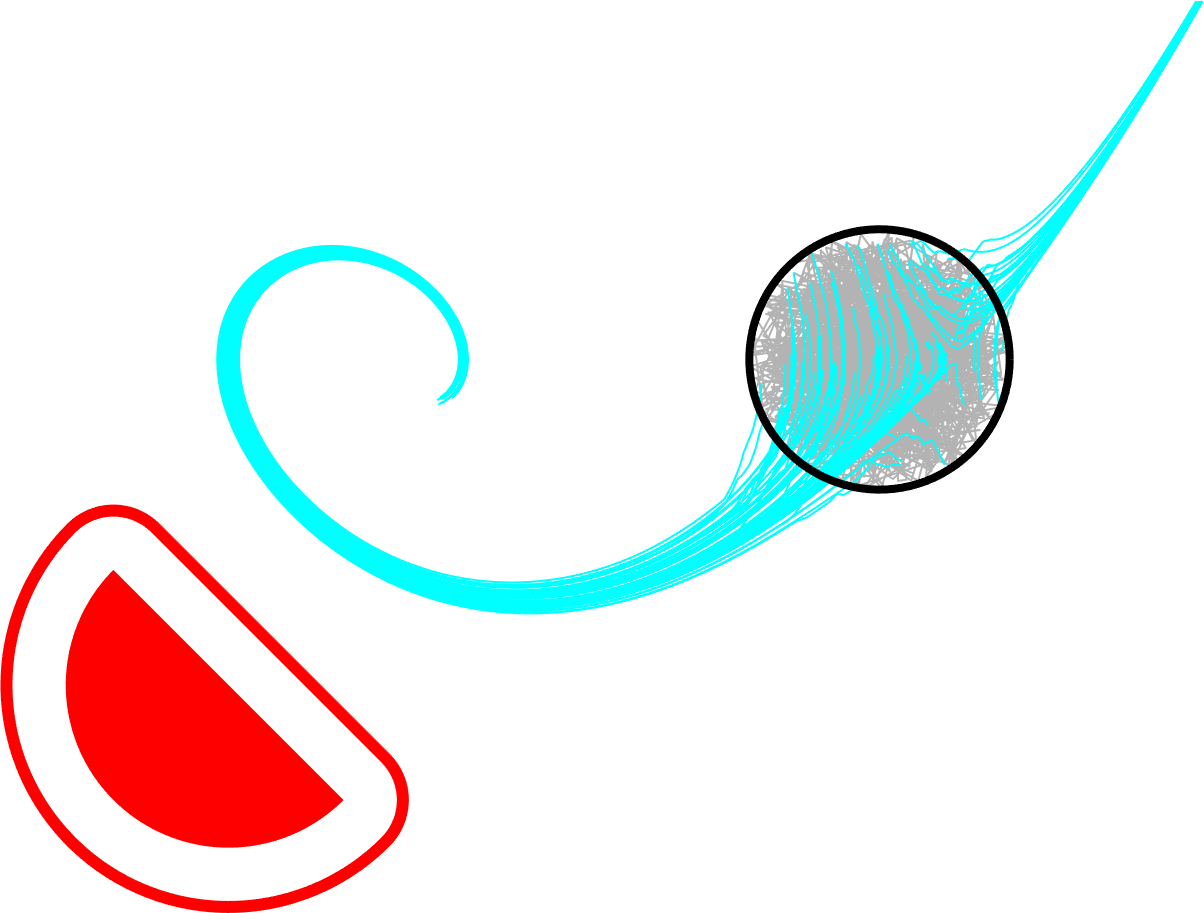}
        \caption{\label{fig:flow_delay_dist} Minimize $c(x; X_u)$ on the delayed Flow system \eqref{eq:flow_delay}}
    \end{figure}

    
    \subsection{Delayed Time-Varying System}
    
    This example involves peak estimation of \iac{DDE} version of the time-varying Example 2.1 of \cite{fantuzzi2020bounding} with
    \begin{equation}
\label{eq:time_var_delay}
    \dot{x}(t) = \begin{bmatrix}x_2(t) t - 0.1 x_1(t) - x_1(t-\tau) x_2(t-\tau)  \\ -x_1(t) t - x_2(t) + x_1(t) x_1(t-\tau)
        \end{bmatrix}.
\end{equation}

The considered support parameters are $\tau = 0.75, \ T = 5, $ and $X = [-1.25, 1.25] \times [-0.75, 1.25]$.
The time-zero set is the disk $X_0 = \{x\in \R^2 \mid (x_1+0.75)^2 + x_2^2 \leq 0.3^2\}$. The only restriction on allowable histories $\hs$ is that their graphs are contained in the history set $H_0 = [-0.75, 0] \times X_0$. 

Solving the \ac{SDP} associated with the \ac{LMI} \eqref{eq:peak_delay_lmi} to maximize the peak function $p = x_1$ yields the sequence of five bounds $p^*_{1:5} = [1.25, 1.25, 1.1978, 0.8543, 0.718264618]$. 
Figure \ref{fig:time_var_peak} plots system trajectories and the $p^*_5$ bound on $x_1$ using the same visual convention as Figure \ref{fig:flow_delay_peak} (black circle $X_0$, grey histories $x_h(t)$, cyan trajectories $x(t \mid x_h)$, red dotted line $x_1 = p^*_5$).

    \begin{figure}[ht]
        \centering
        \includegraphics[width=0.5\linewidth]{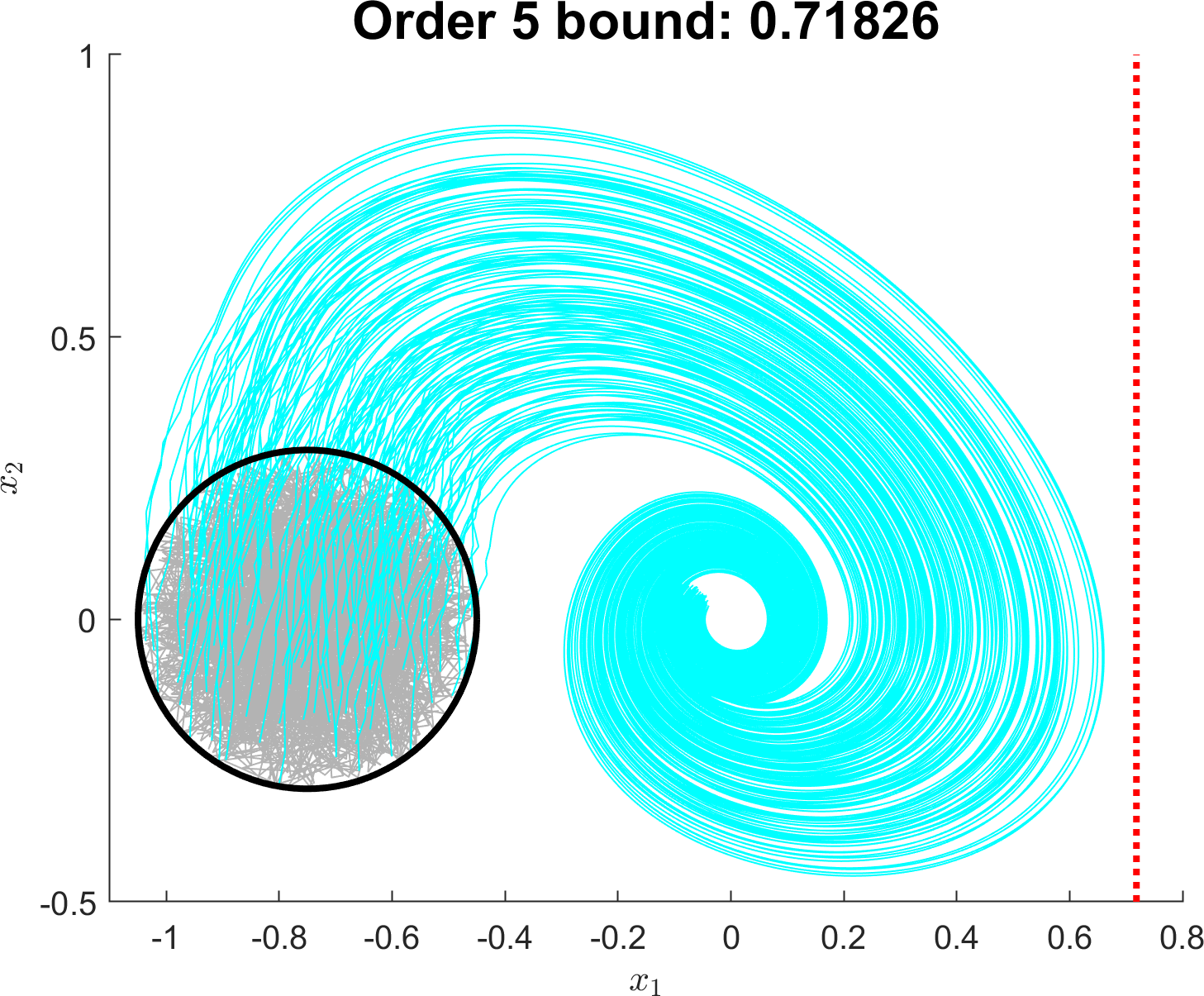}
        \caption{\label{fig:time_var_peak} Maximize $x_1$ on the delayed time-varying \eqref{eq:time_var_delay}}
    \end{figure}

Figure \ref{fig:time_var_peak_3d} plots the corresponding trajectory and bound information in 3d $(t, x_1, x_2)$. The black circles denote the boundary of $H_0$. The history structure inside $H_0$ between times $[-0.75, 1]$ is clearly visible in grey.

        \begin{figure}[ht]
        \centering
        \includegraphics[width=0.6\linewidth]{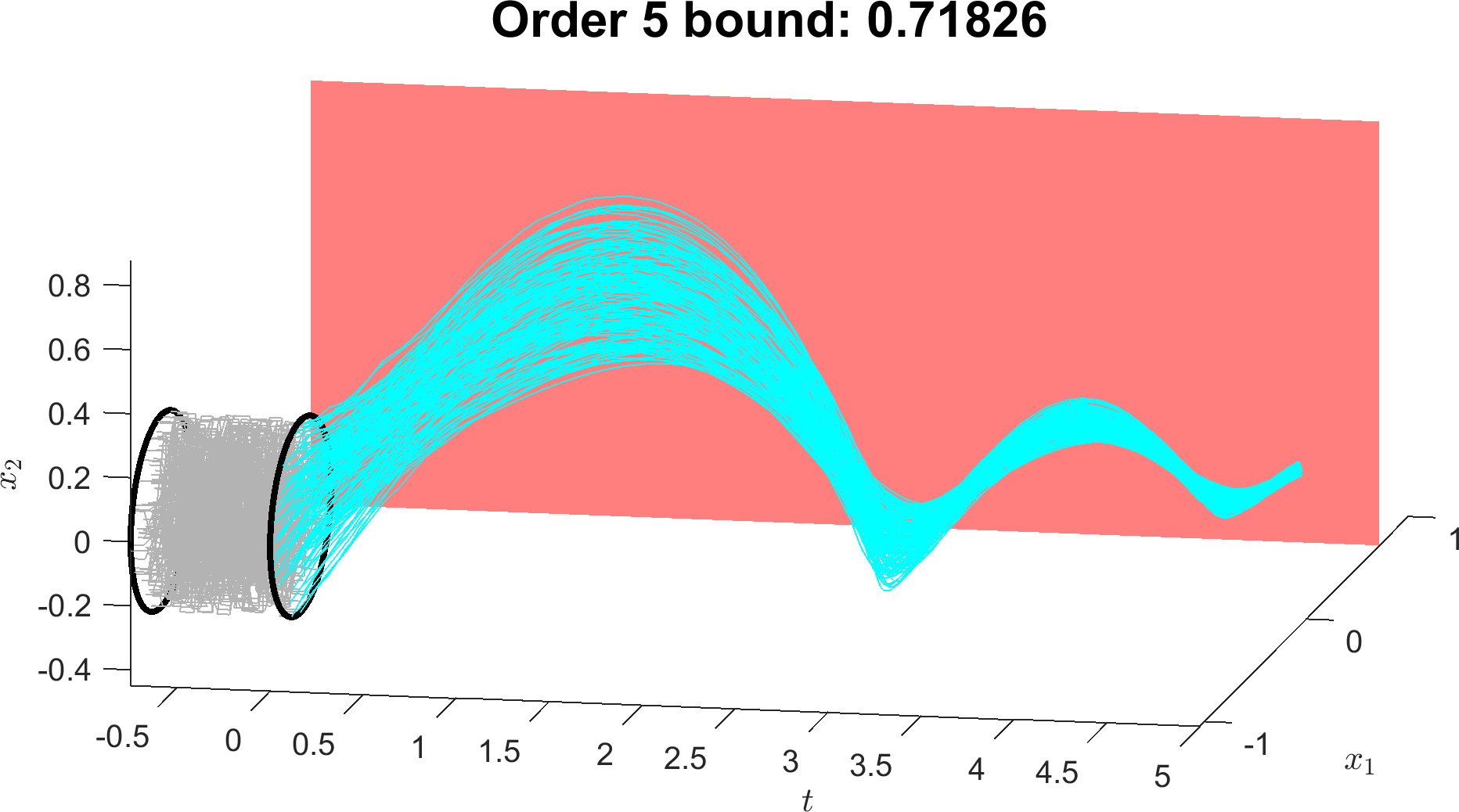}
        \caption{\label{fig:time_var_peak_3d} A 3d plot of \eqref{eq:time_var_delay} and its  $x_1$ bound }
    \end{figure}

The peak estimation of $p = x_2$ under the same system parameters leads to the sequence of five bounds $p^*_{1:5} = [1.25, 1.25, 0.9557, 0.9138, 0.9112].$ 
\section{Extensions}
\label{sec:extensions_delay}

This section discusses several extensions to the \ac{DDE} peak estimation framework.

\subsection{Shaping Constraints}

\label{sec:delay_shaping}
Assumption A5 imposes that the class $\mathcal{H}$ is graph-constrained. 
Some applications involve further structure in the function class $\mathcal{H}$, such as requiring that the histories in $\mathcal{H}$ are constant in time between $t \in [-\tau_r, 0]$. Examples of these constant histories for the system in \eqref{eq:time_delay_fig} staring within the black box ($H_0$) are plotted in Figure \ref{fig:delay_shape_const}. 

\begin{figure}[ht]
    \centering
    \includegraphics[width=0.5\linewidth]{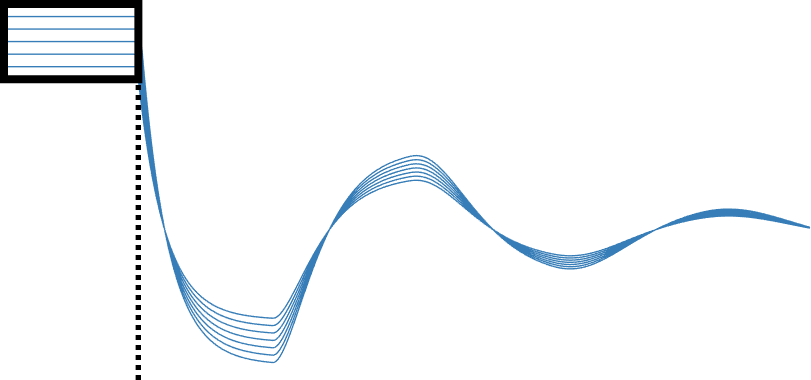}
    \caption{Constant histories in the black box}
    \label{fig:delay_shape_const}
\end{figure}

These types of structure in histories may be realized by adding constraints to $\mu_h$. A method to ensure that the histories in $\mu_h$ are constant in time between $t \in [-\tau, 0]$ is by requiring $\mu_h$ to be the occupation measure of the system $\dot{x}=0$ through a Liouville equation
\begin{align}
    \inp{v(0, x)}{\mu_0} &= \inp{\partial_t v(t, x)}{\mu_h} +\inp{v(-\tau, x)}{\mu_0}& \forall v \in C([-\tau_r, 0]).
    \label{eq:weak_shaping_open}
\end{align}

\subsection{Multiple Time Delays}

\Iac{DDE} with multiple time-delays $0 < \tau_1 < \tau_2 < \ldots < \tau_r$ for $(r, \tau_r)$ finite and a history $x_h \in PC([-\tau_r, 0], X)$ is
\begin{align}
    \dot{x}(t) &= f(t, x(t), x(t-\tau_1), \ldots, x(t - \tau_r)) \label{eq:delay_multiple}\\
    x(t) &= x_h(t), \quad \forall t \in [-\tau_r, 0]. \nonumber
\end{align}

A peak estimation problem for \eqref{eq:delay_multiple} with history class $\mathcal{H}$ and objective $p(x)$ is
\begin{subequations}
\label{eq:peak_delay_multi_traj}
    \begin{align}
    P^* = & \sup_{t^* \in [0, T], \; x_h(\cdot)} p(x(t^* \mid x_h)) &\\
    & \dot{x} =  f(t, x(t), x(t-\tau_1), \ldots x(t-\tau_r)) & & \forall t \in [0, T] \label{eq:delay_dynamics_multi} \\
    & x(t) = x_h(t) & & \forall t \in [-\tau_r, 0]\\
     & x_h(\cdot) \in \mathcal{H}.
    \end{align}
\end{subequations}

A multiple-time-delay \ac{MV}-solution for the peak estimation problem \eqref{eq:peak_delay_multi_traj} is (for  all $i=1..r$):
\begin{subequations}
\label{eq:weak_solution_multi}
    \begin{align}
        &\textrm{History}  & \mu_{h i} &\in \Mp{H_0 \cap ([-\tau_{i}, -\tau_{i-1}] \times X)}\label{eq:weak_solution_multi_history} \\
       & \textrm{Initial} & \mu_0 &\in \Mp{X_0} \\
       &\textrm{Peak}  &\mu_p &\in \Mp{[0, T] \times X} \\
       &\textrm{Time-Slack} & \nu_i &\in \Mp{[0, T] \times X} \label{eq:weak_solution_multi_slack} \\
        &\textrm{Occupation Start }  & \bmu_0 &\in \Mp{[0, T-\tau] \times X^2} \\        
        &\textrm{Occupation End }  & \bmu_i &\in \Mp{[T-\tau_{i}, T-\tau_{i-1}] \times X^2}.
    \end{align}
\end{subequations}

The Lie derivative operator $\Lie_f$ with respect to \eqref{eq:delay_multiple} for $v \in C^1([0, T] \times X)$ is
\begin{equation}
    \Lie_f v(t, x_0) = \partial_t v(t, x_0) + f(t, x_0, x_1, \ldots, x_r) \cdot \nabla_{x_0} v(t, x_0).
\end{equation}

The multiple-time-delay peak estimation \ac{LP} for \eqref{eq:peak_delay_multi_traj} problem of $p(x)$ is

\begin{subequations}
\label{eq:peak_delay_multi_meas}
    \begin{align}
        p^* = & \ \sup \quad \inp{p}{\mu_p} \label{eq:peak_delay_multi_obj} \\
    & \inp{1}{\mu_0} = 1 \label{eq:peak_delay_multi_prob}\\    
    & \pi^{t}_\# \mu_{hi} = \lambda_{[-\tau_{i}, -\tau_{i-1}]} & & \forall i =1..r \label{eq:peak_delay_multi_hist}\\   
    & \mu_p = \delta_0 \otimes\mu_0 + \pi^{t x_0}_\# \Lie_f^\dagger (\bmu_0 + \textstyle \sum_{i=1}^r \bmu_i) \label{eq:peak_delay_multi_flow}\\        
    & \pi^{t x_1}_\# (\bmu_0 + \textstyle \sum_{i=1}^r \bmu_i) + \nu_i = S^{\tau_i}_\#(\sum_{j=1}^i \mu_{h j} + \pi^{t x_0}_\# (\bmu_0 + \sum_{j=1}^{i-1} \bmu_{i} ))  & & \forall i = 1..r\label{eq:peak_delay_multi_cons}\\ 
    & \textrm{Measure Definitions from  \eqref{eq:weak_solution_multi}.} \label{eq:peak_delay_multi_def}
    \end{align}
\end{subequations}

Theorem \ref{thm:delay_upper_bound} can be extended to the multiple-time-delay case to prove that $P^* \leq p^*$ between  \eqref{eq:peak_delay_multi_traj} and \eqref{eq:peak_delay_multi_meas}. Even if Conjecture \ref{conj:delay} holds in the single-delay case, it is unlikely the conjecture is satisfied in the multiple-delay case due to findings in \cite{rosenblueth1992proper}.

\subsection{Uncertainty}

This extension subsection will discuss three types of uncertainty that can affect \acp{DDE} dynamics: time-independent, time-dependent, and unknown-delay.

\subsubsection{Time-independent Uncertainty}

Time-independent uncertainty $\theta \in \Theta$ for a set $\Theta$ can be added to dynamics by adjoining the state $\theta$ following $\dot{\theta}=0$ to \eqref{eq:delay_dynamics}. This same process occurs in \cite{miller2021uncertain} for the \ac{ODE} case.

\subsubsection{Time-dependent Uncertainty}

Time-dependent may be implemented by a Young Measure approach. Given dynamics $\dot{x}(t) = f(t, x(t), x(t-\tau), w(t))$ for $w(t) \in W$, the joint occupation measures representing trajectories are $\bmu_0 \in \Mp{[0, T-\tau] \times X^2 \times W}$ and  $\bmu_1 \in \Mp{[T-\tau, T] \times X^2 \times W}$. No substantial changes are required to the Liouville nor consistency constraints.

Input delays may also be introduced into dynamics with $\dot{x}(t) = f(t, x(t), x(t-\tau), w(t), w(t-\tau))$ under the state history $x_h$ and input history $w_h$ (defined in times $t \in [-\tau, 0]$). The associated joint occupation measures are now $\bmu_0 \in \Mp{[0, T-\tau] \times X^2 \times W^2}$ and   $\bmu_1 \in \Mp{[T-\tau, T] \times X^2 \times W^2}$, each involving variables $(t, x_0, x_1, w_0, w_1)$. The state-input history class is $\hs \in PC([-\tau, 0], X \times W)$, and its history occupation measure now includes an input component $\mu_h \in \Mp{[-\tau, 0] \times X \times W}$.

While the Liouville equation stays the same as \eqref{eq:peak_delay_liou}, the consistency constraint ensures that the $w_1$ coordinate contains a delayed copy of $w_0$,
\begin{align}
    \pi^{t x_1 w_1}_\#(\bmu_0 + \bmu_1) = S^\tau_\#(\mu_h + \pi^{t x_0 w_0} \bmu_0).
\end{align}

\subsubsection{Unknown Delays}

This extension focuses on dynamics where the time-independent (constant) delay is unknown but fixed in the finite range of $\tau \in [\ubar{\tau}, \bar{\tau}]$. The unknown delay $\tau$ must be treated as an additional state with $\dot{\tau} = 0$.

The following sets may be defined:
\begin{subequations}
\begin{align}
    \Omega_{h} &= \{(\tau, t, x) \mid \tau \in [\ubar{\tau}, \bar{\tau}], \ (t, x) \in H_0\mid_\tau \}  \\
    \Omega_{0} &= \{(\tau, t, x_0, x_1) \mid \tau \in [\ubar{\tau}, \bar{\tau}], \ t \in [0, T-\tau], (x_0, x_1) \in X^2 \}\\
    \Omega_{1} &= \{(\tau, t, x_0, x_1) \mid \tau \in [\ubar{\tau}, \bar{\tau}], \ t \in [T-\tau, T], (x_0, x_1) \in X^2 \}.
\end{align}
\end{subequations}

\Iac{MV}-solution in the unknown-delay case has the form:
\begin{subequations}
\label{eq:weak_solution_unknown}
    \begin{align}
        &\textrm{History}  & \mu_h &\in \Mp{\Omega_{h}} \label{eq:weak_solution_unknown_history} \\
        &\textrm{History Slack}  &\bar{\mu}_h &\in \Mp{\Omega_{h}} \label{eq:weak_solution_unknown_history_slack} \\
       & \textrm{Initial} & \mu_0 &\in \Mp{X_0 \times [\ubar{\tau}, \bar{\tau}]} \\
       &\textrm{Peak}  &\mu_p &\in \Mp{[0, T] \times X\times [\ubar{\tau}, \bar{\tau}]} \\
       &\textrm{Time-Slack} & \nu &\in \Mp{[0, T] \times X\times [\ubar{\tau}, \bar{\tau}]} \label{eq:weak_solution_unknown_slack} \\
        &\textrm{Occupation Start }  & \bmu_0 &\in \Mp{\Omega_0} \\        
        &\textrm{Occupation End }  & \bmu_1 &\in \Mp{\Omega_1}.
    \end{align}
\end{subequations}

The Liouville and Consistency constraints in the unknown-delay case are unchanged as compared to the known-delay system (with the new state $\dot{\tau}=0$).

However, the history-validity constraints have the following form,
\begin{subequations}
\label{eq:hist_valid_unknown}
\begin{align}
& \inp{1}{\mu_0} = 1 \label{eq:hist_valid_unknown_mass}\\
&\delta_{t=0} \otimes (\pi^\tau_\# \mu_0) = \delta_{t=-\bar{\tau}} \otimes (\pi^{\tau}_\# \mu_0) + (\partial_t)_\# (\pi^{t \tau}_\# (\mu_h + \hat{\mu}_h)) \label{eq:hist_valid_unknown_const}\\
& \pi^{t}(\mu_h + \hat{\mu}_h) = \lambda_{[-\bar{\tau}, 0]}. \label{eq:hist_valid_unknown_lambda}
\end{align}
\end{subequations}
Constraint \eqref{eq:hist_valid_unknown_mass} ensures that the initial distribution $\mu_0$ is a probability measure. Constraint \eqref{eq:hist_valid_unknown_const} imposes that $\tau(t)$ is constant in time between $t=[-\bar{\tau}, 0]$. Constraint \eqref{eq:hist_valid_unknown_lambda} is a domination term that requires the history $x_h$ to be defined in times $[-\bar{\tau}, 0]$.

It is an open problem to extend consistency constraints and \ac{MV}-solutions towards cases where the delay $\tau(t)$ is time-dependent  (such as $\dot{\tau(t)} \in [-B, B]$).

\section{Conclusion}
\label{sec:conclusion}

This paper presented a formulation of 
\ac{MV}-solutions for free-terminal-time \acp{DDE} with multiple histories (Definition \ref{defn:mv_solution}). These \ac{MV}-solutions are formed by the conjunction of Validity, Liouville and Consistency constraints.
These \ac{MV}-solutions may be used to provide upper bounds on peak estimation problems over \acp{DDE} by Program \eqref{eq:peak_delay_meas}.

A vital area for future work is determining the conditions under which $P^* = p^*$ between \eqref{eq:peak_delay_traj} and \eqref{eq:peak_delay_meas} (Conjecture \ref{conj:delay}).
Other areas for future work include applying \ac{MV}-solutions to other  problems (such as reachable set and positive-invariant set estimation) and formulating delay-dependent \ac{MV}-solutions.

\section*{Acknowledgements}

The authors thank Corbinian Schlosser, Matteo Tacchi, Didier Henrion, Leonid Mirkin, Alexandre Seuret, and Sabine Mondi\'{e} for their discussions about \acp{DDE} and occupation measures.

\bibliographystyle{IEEEtran}
\bibliography{references.bib}

\begin{thebibliography}{10}
\providecommand{\url}[1]{#1}
\csname url@samestyle\endcsname
\providecommand{\newblock}{\relax}
\providecommand{\bibinfo}[2]{#2}
\providecommand{\BIBentrySTDinterwordspacing}{\spaceskip=0pt\relax}
\providecommand{\BIBentryALTinterwordstretchfactor}{4}
\providecommand{\BIBentryALTinterwordspacing}{\spaceskip=\fontdimen2\font plus
\BIBentryALTinterwordstretchfactor\fontdimen3\font minus
  \fontdimen4\font\relax}
\providecommand{\BIBforeignlanguage}[2]{{%
\expandafter\ifx\csname l@#1\endcsname\relax
\typeout{** WARNING: IEEEtran.bst: No hyphenation pattern has been}%
\typeout{** loaded for the language `#1'. Using the pattern for}%
\typeout{** the default language instead.}%
\else
\language=\csname l@#1\endcsname
\fi
#2}}
\providecommand{\BIBdecl}{\relax}
\BIBdecl

\bibitem{hale1971functional}
J.~K. Hale, ``{Functional Differential Equations},'' in \emph{Analytic Theory
  of Differential Equations}.\hskip 1em plus 0.5em minus 0.4em\relax Springer,
  1971, pp. 9--22.

\bibitem{kuang1993delay}
Y.~Kuang, \emph{Delay Differential Equations: with Applications in Population
  Dynamics}.\hskip 1em plus 0.5em minus 0.4em\relax Academic Press, 1993.

\bibitem{bellen2013numerical}
A.~Bellen and M.~Zennaro, \emph{Numerical Methods for Delay Differential
  Equations}.\hskip 1em plus 0.5em minus 0.4em\relax Oxford University Press,
  2013.

\bibitem{fridman2014intro}
E.~Fridman, \emph{Introduction to Time-Delay Systems: Analysis and
  Control}.\hskip 1em plus 0.5em minus 0.4em\relax Birkh\"{a}user Basel, 10
  2014.

\bibitem{cho2002linear}
M.~J. Cho and R.~H. Stockbridge, ``{Linear Programming Formulation for Optimal
  Stopping Problems},'' \emph{SIAM Journal on Control and Optimization},
  vol.~40, no.~6, pp. 1965--1982, 2002.

\bibitem{fantuzzi2020bounding}
G.~Fantuzzi and D.~Goluskin, ``{Bounding Extreme Events in Nonlinear Dynamics
  Using Convex Optimization},'' \emph{SIAM Journal on Applied Dynamical
  Systems}, vol.~19, no.~3, pp. 1823--1864, 2020.

\bibitem{warga1974optimal}
J.~Warga, ``{Optimal Controls with Pseudodelays},'' \emph{SIAM Journal on
  Control}, vol.~12, no.~2, pp. 286--299, 1974.

\bibitem{young1942generalized}
L.~C. Young, ``{Generalized Surfaces in the Calculus of Variations},''
  \emph{Annals of Mathematics}, vol.~43, pp. 84--103, 1942.

\bibitem{warga2014optimal}
J.~Warga, \emph{Optimal control of differential and functional
  equations}.\hskip 1em plus 0.5em minus 0.4em\relax Academic press, 2014.

\bibitem{rosenblueth1991relaxation}
J.~F. Rosenblueth and R.~B. Vinter, ``{Relaxation Procedures for Time Delay
  Systems},'' \emph{Journal of Mathematical Analysis and Applications}, vol.
  162, no.~2, pp. 542--563, 1991.

\bibitem{rosenblueth1992strongly}
J.~F. Rosenblueth, ``Strongly and weakly relaxed controls for time delay
  systems,'' \emph{SIAM Journal on Control and Optimization}, vol.~30, no.~4,
  pp. 856--866, 1992.

\bibitem{rosenblueth1992proper}
J.~Rosenblueth, ``Proper relaxation of optimal control problems,''
  \emph{Journal of Optimization Theory and Applications}, vol.~74, no.~3, pp.
  509--526, 1992.

\bibitem{lewis1980relaxation}
R.~Lewis and R.~Vinter, ``{Relaxation of Optimal Control Problems to Equivalent
  Convex Programs},'' \emph{Journal of Mathematical Analysis and Applications},
  vol.~74, no.~2, pp. 475--493, 1980.

\bibitem{lasserre2009moments}
J.~B. Lasserre, \emph{{Moments, Positive Polynomials And Their}
  {Applications}}, ser. Imperial College Press Optimization Series.\hskip 1em
  plus 0.5em minus 0.4em\relax World Scientific Publishing Company, 2009.

\bibitem{prajna2004safety}
S.~Prajna and A.~Jadbabaie, ``{Safety Verification of Hybrid Systems Using
  Barrier Certificates},'' in \emph{International Workshop on Hybrid Systems:
  Computation and Control}.\hskip 1em plus 0.5em minus 0.4em\relax Springer,
  2004, pp. 477--492.

\bibitem{henrion2008nonlinear}
D.~Henrion, J.~B. Lasserre, and C.~Savorgnan, ``Nonlinear optimal control
  synthesis via occupation measures,'' in \emph{2008 47th IEEE Conference on
  Decision and Control}.\hskip 1em plus 0.5em minus 0.4em\relax IEEE, 2008, pp.
  4749--4754.

\bibitem{papa2009sosdelay}
A.~{Papachristodoulou}, M.~M. {Peet}, and S.~{Lall}, ``{Analysis of Polynomial
  Systems With Time Delays via the Sum of Squares Decomposition},'' \emph{IEEE
  Transactions on Automatic Control}, vol.~54, no.~5, pp. 1058--1064, 2009.

\bibitem{miller2020recovery}
J.~{Miller}, D.~{Henrion}, and M.~{Sznaier}, ``{Peak Estimation Recovery and
  Safety Analysis},'' \emph{IEEE Control Systems Letters}, vol.~5, no.~6, pp.
  1982--1987, 2020.

\bibitem{korda2013inner}
M.~Korda, D.~Henrion, and C.~N. Jones, ``Inner approximations of the region of
  attraction for polynomial dynamical systems,'' \emph{IFAC Proceedings
  Volumes}, vol.~46, no.~23, pp. 534--539, 2013.

\bibitem{magron2017discrete}
V.~Magron, P.-L. Garoche, D.~Henrion, and X.~Thirioux, ``{Semidefinite
  Approximations of Reachable Sets for Discrete-time Polynomial Systems},''
  \emph{SIAM J. Control Optim.}, vol.~57, pp. 2799--2820, 03 2017.

\bibitem{miller2022distance_short}
J.~Miller and M.~Sznaier, ``{Bounding the Distance of Closest Approach to
  Unsafe Sets with Occupation Measures},'' in \emph{2022 61st IEEE Conference
  on Decision and Control (CDC)}, 2022, pp. 5008--5013.

\bibitem{papachristodoulou2005tutorial}
A.~Papachristodoulou and S.~Prajna, ``{A Tutorial on Sum of Squares Techniques
  for Systems Analysis},'' in \emph{Proceedings of the 2005, American Control
  Conference, 2005.}\hskip 1em plus 0.5em minus 0.4em\relax IEEE, 2005, pp.
  2686--2700.

\bibitem{prajna2005methods}
S.~Prajna and A.~Jadbabaie, ``{Methods for Safety Verification of Time-Delay
  Systems},'' in \emph{Proceedings of the 44th IEEE Conference on Decision and
  Control}.\hskip 1em plus 0.5em minus 0.4em\relax IEEE, 2005, pp. 4348--4353.

\bibitem{marx2018entropy}
S.~Marx, T.~Weisser, D.~Henrion, and J.~B. Lasserre, ``A moment approach for
  entropy solutions to nonlinear hyperbolic {PDEs},'' \emph{Mathematical
  Control and Related Fields}, vol.~10, no.~1, pp. 113--140, 2020.

\bibitem{korda2018momentspde}
M.~Korda, D.~Henrion, and J.~B. Lasserre, ``Chapter 10 - {M}oments and convex
  optimization for analysis and control of nonlinear pdes,'' in \emph{Numerical
  Control: Part A}, ser. Handbook of Numerical Analysis, E.~Trélat and
  E.~Zuazua, Eds.\hskip 1em plus 0.5em minus 0.4em\relax Elsevier, 2022,
  vol.~23, pp. 339--366.

\bibitem{magron2020optimal}
V.~Magron and C.~Prieur, ``{Optimal Control of Linear PDEs using Occupation
  Measures and SDP Relaxations},'' \emph{IMA Journal of Mathematical Control
  and Information}, vol.~37, no.~1, pp. 159--174, 2020.

\bibitem{barati2012optimal}
S.~Barati, ``{Optimal Control of Constrained Time Delay Systems},''
  \emph{Advanced Modeling and Optimization}, vol.~14, no.~1, pp. 103--116,
  2012.

\bibitem{suresh2018forced}
R.~Suresh and V.~K. Chandrasekar, ``Influence of time-delay feedback on extreme
  events in a forced {L}i\'enard system,'' \emph{Phys. Rev. E}, vol.~98, p.
  052211, Nov 2018.

\bibitem{sadeghi2021universal}
M.~Sadeghi, J.~M. Greene, and E.~D. Sontag, ``Universal features of epidemic
  models under social distancing guidelines,'' \emph{Annual Reviews in
  Control}, 2021.

\bibitem{tao2011introduction}
T.~Tao, \emph{{An Introduction to Measure Theory}}.\hskip 1em plus 0.5em minus
  0.4em\relax American Mathematical Society Providence, RI, 2011, vol. 126.

\bibitem{tacchi2022convergence}
M.~Tacchi, ``{Convergence of Lasserre’s hierarchy: the general case},''
  \emph{Optimization Letters}, vol.~16, no.~3, pp. 1015--1033, 2022.

\bibitem{henrion2003gloptipoly}
D.~Henrion and J.-B. Lasserre, ``{GloptiPoly: Global Optimization over
  Polynomials with Matlab and SeDuMi},'' \emph{ACM Transactions on Mathematical
  Software (TOMS)}, vol.~29, no.~2, pp. 165--194, 2003.

\bibitem{Lofberg2004}
J.~L{\"{o}}fberg, ``{YALMIP : A Toolbox for Modeling and Optimization in
  MATLAB},'' in \emph{In Proceedings of the CACSD Conference}, Taipei, Taiwan,
  2004.

\bibitem{mosek92}
\BIBentryALTinterwordspacing
M.~ApS, \emph{The MOSEK optimization toolbox for MATLAB manual. Version 9.2.},
  2020. [Online]. Available:
  \url{https://docs.mosek.com/9.2/toolbox/index.html}
\BIBentrySTDinterwordspacing

\bibitem{lauer2020incubation}
S.~A. Lauer, K.~H. Grantz, Q.~Bi, F.~K. Jones, Q.~Zheng, H.~R. Meredith, A.~S.
  Azman, N.~G. Reich, and J.~Lessler, ``The incubation period of coronavirus
  disease 2019 (covid-19) from publicly reported confirmed cases: estimation
  and application,'' \emph{Annals of internal medicine}, vol. 172, no.~9, pp.
  577--582, 2020.

\bibitem{miller2021uncertain}
J.~Miller, D.~Henrion, M.~Sznaier, and M.~Korda, ``{Peak Estimation for
  Uncertain and Switched Systems},'' in \emph{2021 60th IEEE Conference on
  Decision and Control (CDC)}, 2021, pp. 3222--3228.

\bibitem{ockendon1971dynamics}
J.~R. Ockendon and A.~B. Tayler, ``The dynamics of a current collection system
  for an electric locomotive,'' \emph{Proceedings of the Royal Society of
  London. A. Mathematical and Physical Sciences}, vol. 322, no. 1551, pp.
  447--468, 1971.

\bibitem{carr19762}
J.~Carr and J.~Dyson, ``2.-the matrix functional differential equation y'(x)=
  ay (lambda x)+ by (x),'' \emph{Proceedings of the Royal Society of Edinburgh
  Section A: Mathematics}, vol.~75, no.~1, pp. 5--22, 1976.

\bibitem{iserles1991generalized}
A.~Iserles, \emph{On the generalized pantograph functional-differential
  equation}.\hskip 1em plus 0.5em minus 0.4em\relax University of Cambridge,
  Department of Applied Mathematics and Theoretical, 1991.

\bibitem{iserles1994nonlinear}
------, ``On nonlinear delay differential equations,'' \emph{Transactions of
  the American Mathematical Society}, vol. 344, no.~1, pp. 441--477, 1994.

\bibitem{liu1996asymptotic}
Y.~Liu, ``Asymptotic behaviour of functional-differential equations with
  proportional time delays,'' \emph{European Journal of Applied Mathematics},
  vol.~7, no.~1, pp. 11--30, 1996.

\bibitem{shampine2005solving}
L.~Shampine, ``Solving odes and ddes with residual control,'' \emph{Applied
  Numerical Mathematics}, vol.~52, no.~1, pp. 113--127, 2005.

\bibitem{ali2009spectral}
I.~Ali, H.~Brunner, and T.~Tang, ``A spectral method for pantograph-type delay
  differential equations and its convergence analysis,'' \emph{Journal of
  Computational Mathematics}, pp. 254--265, 2009.

\bibitem{sedaghat2012numerical}
S.~Sedaghat, Y.~Ordokhani, and M.~Dehghan, ``Numerical solution of the delay
  differential equations of pantograph type via chebyshev polynomials,''
  \emph{Communications in Nonlinear Science and Numerical Simulation}, vol.~17,
  no.~12, pp. 4815--4830, 2012.

\bibitem{bahcsi2015numerical}
M.~M. Bah{\c{s}}i and M.~{\c{C}}evik, ``Numerical solution of pantograph-type
  delay differential equations using perturbation-iteration algorithms,''
  \emph{Journal of Applied Mathematics}, 2015.

\bibitem{tacchi2021thesis}
M.~Tacchi, ``Moment-sos hierarchy for large scale set approximation.
  application to power systems transient stability analysis,'' Ph.D.
  dissertation, Toulouse, INSA, 2021.

\bibitem{plaksin2020minimax}
A.~Plaksin, ``{Minimax and Viscosity Solutions of Hamilton--Jacobi--Bellman
  Equations for Time-Delay Systems},'' \emph{Journal of Optimization Theory and
  Applications}, vol. 187, no.~1, pp. 22--42, 2020.

\bibitem{santos2009linear}
O.~Santos, S.~Mondi{\'e}, and V.~L. Kharitonov, ``Linear quadratic suboptimal
  control for time delays systems,'' \emph{International Journal of Control},
  vol.~82, no.~1, pp. 147--154, 2009.

\bibitem{ortega2021comments}
J.-M. Ortega-Mart{\'\i}nez, O.-J. Santos-S{\'a}nchez, and S.~Mondi{\'e},
  ``Comments on the bellman functional for linear time-delay systems,''
  \emph{Optimal Control Applications and Methods}, vol.~42, no.~5, pp.
  1531--1540, 2021.

\bibitem{liberzon2011calculus}
D.~Liberzon, \emph{{Calculus of Variations and Optimal Control Theory: A
  Concise Introduction}}.\hskip 1em plus 0.5em minus 0.4em\relax Princeton
  University Press, 2011.

\bibitem{jones2021polynomial}
M.~Jones and M.~M. Peet, ``{Polynomial Approximation of Value Functions and
  Nonlinear Controller Design with Performance Bounds},'' 2021.

\end{thebibliography}

\appendix

\section{Delay Structures}
\label{app:delay_structure}
This chapter has focused on supremizing $p(x)$ in \eqref{eq:peak_delay_traj} over continuous-time systems with a discrete delay $x(t-\tau)$.
This subsection will discuss peak estimation of $p(x)$ with respect to other types of dynamics and delay structures.

\subsection{Proportional Time-Delays}

A system with a proportional delay is defined with respect to a scaling term $\kappa \in [0, 1)$:
\begin{align}
    \dot{x}(t) = f(t, x(t), x(\kappa t)). \label{eq:dynamics_prop}
\end{align}

Proportional time delays are observed in the current collection of a pantograph on a streetcar \cite{ockendon1971dynamics}. References on functional differential equation with proportional time delay include \cite{carr19762, iserles1991generalized, iserles1994nonlinear, liu1996asymptotic}.
MATLAB uses the command \texttt{ddesd} to solve DDEs with time-dependent delays by an RK4 algorithm \cite{shampine2005solving}. 
Other numerical algorithms specifically for proportional delays include \cite{ali2009spectral, sedaghat2012numerical, bahcsi2015numerical}. 

The peak estimation problem over \eqref{eq:dynamics_prop} is,
\begin{subequations}
\label{eq:peak_delay_prop_traj}
    \begin{align}
    P^* = & \sup_{t^* \in [0, T], \; x_0 \in X_0} p(x(t^* \mid x_0)) &\\
    & \dot{x} =  f(t, x(t), x(\kappa t)) & & \forall t \in [0, T] \label{eq:delay_dynamics_prop}.
    \end{align}
\end{subequations}
\Iac{MV}-solution for proportional time delays is
\begin{subequations}
\label{eq:weak_solution_prop}
    \begin{align}
       & \textrm{Initial} & \mu_0 &\in \Mp{X_0} \\
       &\textrm{Peak}  &\mu_p &\in \Mp{[0, T] \times X} \\
       &\textrm{Time-Slack} & \nu &\in \Mp{[0, T] \times X} \label{eq:weak_solution_slack_prop} \\
        &\textrm{Occupation Start }  & \bmu_0 &\in \Mp{[0, \kappa T] \times X^2} \label{eq:weak_solution_start_prop}
         \\        
        &\textrm{Occupation End }  & \bmu_1 &\in \Mp{[\kappa T, T] \times X^2}. \label{eq:weak_solution_end_prop}
    \end{align}
\end{subequations}

Note how the \ac{MV}-solution \eqref{eq:weak_solution_prop} lacks a history measure $\mu_h$ as compared with \eqref{eq:weak_solution}, and also how the limits on \eqref{eq:weak_solution_start_prop}-\eqref{eq:weak_solution_end_prop} are $[0, \kappa T]$ and $[\kappa T, T]$ respectively.

The Lie derivative operator $\Lie$ with $\Lie v= (\partial_t + f(t, x_0, x_1) \cdot \nabla_{x_0}) v(t, x_0)$ is the same as in the discrete-delay case \eqref{eq:peak_delay_liou}, but under dynamics \eqref{eq:dynamics_prop}.

The consistency constraint follows from a modification of Lemma \ref{lem:consistency_free}:
\begin{lem}
\label{lem:consistency_free_prop}
    Let $x(\cdot)$ be a solution to \eqref{eq:delay_dynamics_prop} with an initial condition of $x_0 \in X_0$ and a stopping time of $t^* \in [0, T]$. The following pairs of integral are equal for all  $\phi \in C([0, T] \times X)$:
    \begin{align}
    \label{eq:change_limits_free_prop}
        &\int_{0}^{ t^*} \phi(t, x(\kappa t))dt = \frac{1}{\kappa} \int_{0}^{\min(t^*/\kappa, T)} \phi(t'/\kappa, x(t)) dt'.
    \end{align}
\end{lem}
\begin{proof}
   This relation is due to a change of variable with $t' \leftarrow \kappa t$.
\end{proof}

The resultant consistency constraint w.r.t. the measures in \eqref{eq:weak_solution_prop} is
\begin{align}
\label{eq:consistency_prop}
\inp{\phi(t, x_1)}{\bmu_0 + \bmu_1} + \inp{\phi(t, x)}{\nu} &= \inp{\phi(t/\kappa, x_0) / \kappa}{\bmu_0}.
\end{align}

Expressing the linear expansion operator $E_\kappa$ as $E_\kappa \phi(t, x) = \phi(t/\kappa, x_0)\kappa$, the measure \ac{LP} for problem \eqref{eq:peak_delay_prop_traj} is,
\begin{subequations}
\label{eq:peak_delay_prop_meas}
    \begin{align}
    p^* = & \ \sup \quad \inp{p}{\mu_p} \label{eq:peak_delay_prop_obj} \\    
    & \inp{1}{\mu_0} = 1 \label{eq:peak_delay_prop_prob}\\    
    & \mu_p = \delta_0 \otimes\mu_0 + \pi^{t x_0}_\# \Lie_f^\dagger (\bmu_0 + \bmu_1) \label{eq:peak_delay_prop_flow}\\     
    & \pi^{t x_1}_\# (\bmu_0 + \bmu_1) + \nu = E^\kappa_\#(\pi^{t x_0}_\# \bmu_0) \label{eq:peak_delay_prop_cons}\\ 
    & \textrm{Measure Definitions from  \eqref{eq:weak_solution_prop}.} \label{eq:peak_delay_prop_def}
    \end{align}
\end{subequations}

Problem \eqref{eq:peak_delay_prop_meas} upper-bounds \eqref{eq:peak_delay_prop_traj} by following the reasoning from Theorem \ref{thm:delay_upper_bound} for the proportional-delay case.

\begin{rmk}
Proportional and discrete delays can be applied together to form dynamics,
\begin{align}
\label{eq:prop_combined}
    \dot{x}(t) = f(t, x(t), x(\kappa t - \tau).
\end{align}
Causalness of \eqref{eq:prop_combined} requires that $\kappa \in [0, 1)$ and $\tau \geq 0$. A consistency constraint may be posed using an integral relation
\begin{align}
    \int_{0}^T \phi(t, x(\kappa t - \tau)) dt = \int_{-\tau}^{\kappa T - \tau} \phi((t+\tau)/\kappa, x(t))/\kappa dt,
\end{align}
used as a step towards forming Lemmas \ref{lem:consistency_free} and \ref{lem:consistency_free_prop}.
\end{rmk}

\subsection{Discrete-Time Systems}

This subsection will concentrate on a discrete-time system with a long time delay. The discrete-time system $x[t]$ is defined w.r.t. a delay $\tau \in \N$, and a time horizon $T \in \N$ under the assumption that $\tau < T$.

The peak estimation program for a system with discrete-time dynamics and one time delay $\tau$ is:
\begin{subequations}
\label{eq:peak_delay_disc_traj}
    \begin{align}
    P^* = & \sup_{t^* \in 0..T, \; x_h[\cdot]} p(x[t \mid x_h]) &\\
    & \dot{x} =  f(t, x[t], x[t-\tau]) & & \forall t \in 1..T \label{eq:delay_dynamics_disc} \\
    & x[t] = x_h[t] & & \forall t \in [-\tau, 0]\\
     & x_h[\cdot] \in \mathcal{H}.
    \end{align}
\end{subequations}

Delayed dynamics \eqref{eq:delay_dynamics_disc} may be implemented as a non-delayed discrete system by state inflation in terms of $x[t-(0..\tau)]$ \cite{fridman2014intro}. Such state augmentation could lead to a large number of variables in systems analysis and result in intractably large computational problems.

This subsection will define \iac{MV}-solution using the variables from \eqref{eq:weak_solution}, in which the measures with the maximum number of variables are $(\bmu_0(t, x_0, x_1), \bmu_1(t, x_0, x_1))$.

\subsubsection{History-Validity}

The history-validity constraint for discrete-time systems will separate the history $x_h[t]$ into a time-zero component ($\mu_0$) and a history component $t \in -\tau..-1$ ($\mu_h$). The time-zero component is $\mu_0 \in X_0$, as in Section \ref{sec:history_validity}.

The history measure $\mu_h$ should represent a history $x_h[t]$ defined between $t \in -\tau..-1$. This may be imposed by setting the $t$-marginal of $\mu_h$ to a train of Dirac-deltas supported at sample times $-\tau..-1$,
\begin{align}
    \pi^{t}_\# \mu_h &= \textstyle \sum_{t=-\tau}^1 \delta_{t}.
\end{align}

\subsubsection{Liouville}

The discrete-time Liouville equation \cite{magron2017discrete} applied to the dynamics \eqref{eq:delay_dynamics_disc} for all test functions $v \in C([0, T+1] \times X)$ is,
\begin{align}
    \inp{v(t, x)}{\mu_p} &= \inp{v(0, x)}{\mu_0} + \inp{v(t+1, f(t, x_0, x_1) - v(t, x_0, x_1)}{\bmu_0 + \bmu_1}. \label{eq:delay_disc_push_int}
    \intertext{The Liouville constraint in \eqref{eq:delay_disc_push_int} will be abbreviated (using the identity operator $Id(x)=x$),}
    \mu_p &= \delta_0 \otimes \mu_0 + \pi^{t x_0}_\# ((t+1, f)_\# - Id_\#)(\bmu_0 + \bmu_1).\label{eq:delay_disc_push}
\end{align}

\subsubsection{Consistency}

The consistency constraint for dynamics \eqref{eq:delay_dynamics_disc} may be derived from the following Lemma,
\begin{lem}
\label{lem:consistency_free_disc}
    Let $x[\cdot]$ be a trajectory of \eqref{eq:delay_dynamics_disc} given an initial history $x_h$ and a stopping time of $t^* \in 0..T$. It follows that the below pair of summations are equal for all  $\phi \in C([0, T] \times X)$:
    \begin{align}
    \label{eq:change_limits_free_discrete}
        &\left(\sum_{t=0}^{t^*} + \sum_{t=t^*}^{\min (T, t^* + \tau)} \right) \phi(t, x[t-\tau])dt = \sum_{t'=-\tau}^{\min(T-\tau, t^*)} \phi(t'+\tau, x[t]).
    \end{align}
\end{lem}
\begin{proof}
   The index of summation is exchanged as $t' \rightarrow t-\tau$.
\end{proof}

The resultant consistency constraint from Lemma \ref{lem:consistency_free_disc} has an identical form as  \eqref{eq:consistency_free} with
\begin{align}
    \pi^{t x_1}_\#(\bmu_0 + \bmu_1) + \nu = S^\tau_\#(\mu_h + \pi^{t x_0} \bmu_0).
\end{align}

\subsubsection{Measure Program}

The peak estimation measure \ac{LP} that upper-bounds \eqref{eq:peak_delay_disc_traj} is,
\begin{subequations}
\label{eq:peak_delay_disc_meas}
    \begin{align}
        p^* = & \ \sup \quad \inp{p}{\mu_p} \label{eq:peak_delay_disc_obj} \\
    & \inp{1}{\mu_0} = 1 \label{eq:peak_delay_disc_prob}\\    
    & \pi^{t}_\# \mu_h = \textstyle \sum_{t'=-\tau}^{-1} \delta_{t=t'} \label{eq:peak_delay_disc_hist}\\   
    & \mu_p = \delta_0 \otimes \mu_0 + \pi^{t x_0}_\# ((t+1, f, x_1)_\# - Id_\#)(\bmu_0 + \bmu_1)\label{eq:peak_delay_disc_flow}\\        
    & \pi^{t x_1}_\# (\bmu_0 + \bmu_1) + \nu = S^\tau_\#(\mu_h + \pi^{t x_0}_\# \bmu_0) \label{eq:peak_delay_disc_cons}\\ 
    & \textrm{Measure Definitions from  \eqref{eq:weak_solution}.} \label{eq:peak_delay_disc_def}
    \end{align}
\end{subequations}

This upper-bound also follows from constructing measures from trajectories as in Theorem 
\ref{thm:delay_upper_bound}.
\section{Proof of Strong Duality}
\label{app:strong_duality_delay}

This appendix will prove strong duality between programs \eqref{eq:peak_delay_meas} and \eqref{eq:peak_delay_cont} for peak estimation. The general pattern for cone conventions and affine maps of \cite{tacchi2021thesis} will be followed.


The signed measure spaces of \eqref{eq:weak_solution} are
\begin{align}
     \mathcal{X} = &\mathcal{M}(H_0) \times \mathcal{M}(X_0) \times \mathcal{M}([0, T]\times X)^2 \nonumber \\
     &\times \mathcal{M}([0, T-\tau] \times X^2) \times \mathcal{M}([T-\tau, T] \times X^2) \label{eq:dual_spaces_delay} \\
     \mathcal{X}' = &C(H_0) \times C(X_0) \times C([0, T]\times X)^2 \nonumber \\
     &\times C([0, T-\tau] \times X^2) \times C([T-\tau, T] \times X^2). \nonumber
\end{align}

Their nonnegative subcones (with \eqref{eq:weak_solution} membership) are topological duals under A1 and have definitions
\begin{align}
     \mathcal{X}_+ = &\mathcal{M}_+(H_0) \times \mathcal{M}_+(X_0) \times \mathcal{M}_+([0, T]\times X)^2 \nonumber  \\
     & \times \mathcal{M}_+([0, T-\tau] \times X^2) \times \mathcal{M}_+([T-\tau, T] \times X^2) \label{eq:dual_cones_delay} \\
     \mathcal{X}'_+= &C_+(H_0) \times C_+(X_0) \times C_+([0, T]\times X)^2  \nonumber \\
    &\times C_+([0, T-\tau] \times X^2) 
      \times C_+([T-\tau, T] \times X^2). \nonumber
\end{align}

The collection of measures in \eqref{eq:weak_solution} will be denoted as $\boldsymbol{\mu} = (\mu_h, \mu_0, \mu_p, \nu, \bmu_0, \bmu_1)$ and is a member of $\mathcal{X}_+$. 

The constraint spaces of \eqref{eq:peak_delay_prob}-\eqref{eq:peak_delay_cons} are
\begin{align}
    \mathcal{Y} &= \R \times C([-\tau, 0]) \times  C^1([0, T] \times X) \times C([0, T] \times X) \\
    \mathcal{Y}' &= 0 \times \mathcal{M}([-\tau, 0])  \times C^1([0, T] \times X)' \times  \mathcal{M}([0, T] \times X).
\end{align}
The space $\mathcal{X}$ has the weak-* topology and 
$\mathcal{Y}$ has a sup-norm bounded weak topology.
Because there are no affine-inequality constraints present in \eqref{eq:peak_delay_prob}-\eqref{eq:peak_delay_cons}, we write $\mathcal{Y}_+ = \mathcal{Y}$ and $\mathcal{Y}_+' = \mathcal{Y}'$ to match the notation used in \cite{tacchi2021thesis}. 

The variables of \eqref{eq:peak_delay_cont} with $\boldsymbol{\ell} = (\gamma, \xi,  v, \phi)$ satisfy $\boldsymbol{\ell} \in \mathcal{Y}_+'$.

A pair of adjoint linear operators $\A: \mathcal{X}_+ \rightarrow \mathcal{Y}_+$ and $\A': \mathcal{Y}_+' \rightarrow \mathcal{X}_+'$  induced from \eqref{eq:peak_delay_prob}-\eqref{eq:peak_delay_cons} are,
\begin{align}
    \A(\boldsymbol{\mu}) &=\begin{bmatrix}\inp{1}{\mu_0}\\ \pi^t_\# \mu_h \\ \mu_p- \delta_0 \otimes\mu_0 -\Lie_f^\dagger \mu  \\ S^\tau_\#(\mu_h + \pi^{t x_0}_\# \bmu_0)- \pi^{t x_1}_\# (\bmu_0 + \bmu_1) - \nu  \end{bmatrix}\\ 
    \A'(\boldsymbol{\ell}) &=\begin{bmatrix}\xi(t) + \phi(t+\tau, x)\\ \gamma - v(0, x) \\ v(t, x) \\ -\phi(t, x) \\ -\Lie_f v(t, x_0) - \phi(t, x_1) + \phi(t+\tau, x_0) \\ -\Lie_f v(t, x_0) - \phi(t, x_1) \end{bmatrix}.\nonumber
\end{align}
The cost and answer vectors are
\begin{align}
    \mathbf{c} &= [0, 0, p, 0, 0, 0] \\
    \mathbf{b} &= [1, \lambda_{[-\tau, 0]}, 0, 0].
\end{align}
Problem \ref{eq:peak_delay_meas} may be expressed as the standard-form \ac{LP}:
\begin{align}
    p^* =& \sup_{\boldsymbol{\mu} \in \mathcal{X}_+}\inp{\mathbf{c}}{\boldsymbol{\mu}} = \inp{p}{\mu_p}, & & \mathbf{b} - \A(\boldsymbol{\mu}) \in \mathcal{Y}_+. \label{eq:peak_delay_meas_std}\\
\intertext{The standard-form dualization of \eqref{eq:peak_delay_meas_std}  is }
    d^* = &\inf_{\boldsymbol{\ell} \in \mathcal{Y}'_+} \inp{\boldsymbol{\ell}}{\mathbf{b}} = \gamma + \textstyle \int_{t=-\tau}^0 \xi(t) dt,
    & &\A'(\boldsymbol{\ell}) - \mathbf{c} \in \mathcal{X}_+. \label{eq:peak_delay_cont_std}
\end{align}
The standard-form \eqref{eq:peak_delay_cont_std} may be expanded into \eqref{eq:peak_delay_cont}.

Given that all sets are compact (A1), measures in $\boldsymbol{\mu}$ are bounded (Lemma \ref{lem:moment_bound}), functions in $(c, \mathcal{A})$ are continuous (A2, A4, $v\in C^1([0, T] \times X) \implies \Lie_f v \in C([0, T] \times X)$ ), and there exists a feasible measure solution (Theorem \ref{thm:delay_upper_bound}); it holds that strong duality between \eqref{eq:peak_delay_meas} and \eqref{eq:peak_delay_cont} is achieved by Theorem 2.6 of \cite{tacchi2021thesis}.





\section{Subvalue Functionals and DDE Control}
\label{app:ocp}

This appendix analyzes \iac{DDE} \ac{OCP} when posed over a given history $x_h(\cdot)$. The function $J(t, x_0, x_1, u)$ is a running cost evaluated on the trajectory starting from $x_h$, and $J_T(x)$ is a terminal cost at time $T$. The final point $x(T)$ must reside in the terminal set $X_T \subseteq X$. The controller $u(\cdot)$ must reside inside the compact set $U \subset \R^m$ at each time. The  \ac{DDE} \ac{OCP} under these constraints is
\begin{subequations}
    \label{eq:u_traj}
    \begin{align}
    P^* = & \inf_{u(t)}  \int_{t=0}^T{J(t, x(t), x(t-\tau), u(t)) dt} + J_T(x(T))\\
    & \dot{x} =  f(t, x(t), x(t -\tau), u(t)) & & \forall t \in [0, T] \label{eq:dynamics_delay_u}\\
    & u(t) \in U& & \forall  t \in [0, T]  \\
    & x(t) = x_h(t) & & \forall t \in [-\tau, 0] \\
    & x(T) \in X_T.
    \end{align}
\end{subequations}

Problem \ref{eq:u_traj} was addressed in works such as \cite{warga1974optimal, rosenblueth1991relaxation, rosenblueth1992proper, rosenblueth1992strongly, plaksin2020minimax}, and was completely solved in the case of linear \ac{DDE} dynamics and a quadratic objective in \cite{santos2009linear, ortega2021comments}.

A pair of infinite-dimensional \acp{LP} are synthesized to bound the \ac{OCP} in \eqref{eq:u_traj}. 

This appendix assumes that the terminal time $T$ is fixed to simplify analysis.  
The \ac{MV}-solution from Section \ref{sec:peak_lp} involves a free terminal time, multiple histories, and zero running cost $(J=0)$.

\subsection{Control Measure Program}

The deterministic control law $u(t, x_0, x_1)$ at each time $t$ is relaxed into a probability distribution $\xi_u(u \mid t, x_0, x_1)$  \cite{young1942generalized}. 

The measures involved in \iac{MV}-solution of \eqref{eq:u_traj} are
\begin{subequations}
\label{eq:weak_solution_u}
    \begin{align}
       & \textrm{Initial} & \mu_0 &\in \Mp{X_0} \\
       &\textrm{Peak}  &\mu_p &\in \Mp{[0, T] \times X} \\
        &\textrm{Occupation Start }  & \bmu_0 &\in \Mp{[0, \kappa T] \times X^2 \times U} \label{eq:weak_solution_start_u}
         \\        
        &\textrm{Occupation End }  & \bmu_1 &\in \Mp{[\kappa T, T] \times X^2 \times U}. \label{eq:weak_solution_end_u}
    \end{align}
\end{subequations}

The time-slack measure is set to $\nu=0$ because of the fixed-terminal-time setting ($t^*=T$). The symbol $\mu_{x_h(\cdot)}$ is the occupation measure of $t \mapsto (t, x_h(t))$ between $t \in [-\tau, 0]$.
The history measure $\mu_h$ from \eqref{eq:weak_solution} is set equal to $\mu_{x_h(\cdot)}$ in the case of a single history. Similarly, the initial measure $\mu_0$ is set to the Dirac delta $\delta_{x=x_h(0^+)}$.

A measure relaxation to the optimal program in \eqref{eq:u_traj} is
\begin{subequations}
\label{eq:u_meas}
\begin{align}
p^* =  & \inf \quad \inp{J}{\bar{\mu}} + \inp{J_T}{\mu_T} \\
    & \delta_T \otimes \mu_T = \delta_{t=0, x=x_h(0^+)} + \pi^{t x_0}_\# \Lie_f^\dagger (\bmu_0 + \bmu_1)\\
    & \pi^{t x_1}_\# (\bmu_0 + \bmu_1) = S^\tau_\#( \mu_{x_h(\cdot)} + \pi^{t x_0}_\# \bmu_0)\\
    & \text{Measures from \eqref{eq:weak_solution_u}.}
\end{align}
\end{subequations}

This measure relaxation is based on the optimal control framework of \cite{lewis1980relaxation, henrion2008nonlinear}. Young measure formulations for \ac{DDE} \ac{OCP} have been developed in \cite{warga1974optimal, rosenblueth1991relaxation, rosenblueth1992proper, rosenblueth1992strongly, warga2014optimal}, but Liouville equations began use only in \cite{barati2012optimal}.

\subsection{Control Function Program}

Dual variables $v \in C^1([0, T] \times X)$ and $\phi \in C([0, T] \times X)$ may be introduced to form the dual program of \eqref{eq:u_meas}:
\begin{subequations}
\label{eq:u_cont}
\begin{align}
    d^* = & \ \sup \quad v(0, x_h(0)) +\textstyle \textstyle\int_{-\tau}^{0} \phi(t + \tau, x_h(t)) dt\label{eq:u_cont_obj}& \\
    &J_T(x) - v(T, x) \geq 0 \label{eq:u_cont_term} & & \forall x \in X_T \\
    & \Lie_f v + J(t, x_0, u) -  \phi(t, x_1)  +  \phi(t+\tau, x_0) \geq 0 \label{eq:u_cont_f} & & \forall (t, x_0, x_1 ,u) \in [0, T-\tau] \times X^2 \times U\\
    & \Lie_f v + J(t, x_0, u) -  \phi(t, x_1) \geq 0 \label{eq:u_cont_f_end} & &  \forall (t, x_0, x_1 ,u) \in [T-\tau, T] \times X^2 \times U\\
    &v \in C^1([0, T] \times X) & \label{eq:u_cont_v} \\
    &\phi \in C([0, T] \times X)  \label{eq:u_cont_phi}.
\end{align}
\end{subequations}

This dual program is obtained (with strong duality) by following nearly identical steps to Appendix \ref{app:strong_duality_delay}.

\subsection{True Value Functional}
\label{sec:value_functional}

Let $\mathcal{U}$ be the admissible class of control inputs (such as $\mathcal{U} = \{u : [0, T] \rightarrow U\}$).
Given a time $t \in [0, T]$, a current state $z \in X$, and a history $w(\cdot) \in PC([-\tau, 0], X)$, the value functional $V^*$ associated with the \ac{OCP} \eqref{eq:u_traj} is

\begin{align}
    V^*(t, z, w(\cdot)) = & \min_{u \in \mathcal{U}} \int_{t'=t}^T{J(t', x(t' \mid t, w, u), u(t')) dt'} + J_T(x(T \mid t,  w, u)) \nonumber\\
    & \dot{x} =  f(t, x(t), x(t - \tau), u(t)) & & \forall t \in [0, T] \nonumber\\
    & x(t') = w(t') & & \forall t \in [t-\tau, t) \nonumber\\
    & x(t) = z \\
        & x(T) \in X_T \nonumber \\
    & u(t) \in U& & \forall  t \in [0, T]. \nonumber
\end{align}

The value functional $V^*$ is the cost of solving Problem \eqref{eq:u_traj} starting at time $t$ and state $z$ with history $w(\cdot)$. The convention of arguments $t, z, w(\cdot)$ was taken from \cite{plaksin2020minimax}.

The \ac{HJB} equation of optimality is \cite{liberzon2011calculus}
\begin{subequations}
\label{eq:hjb}
\begin{align}
    0 &= J_T(x(T)) - V^*(T, x(T)) \label{eq:hjb_term_true}\\
    0 &= \inf_{u \in U} \left(\dot{V}^*(t, x(t), x_{\tau}(\cdot), u) + J(t, x(t), u(t)) \right) & \forall t \in [0, T].  \label{eq:hjb_traj_true}
\end{align}
\end{subequations}

The Cauchy problem of (7, 8) in \cite{plaksin2020minimax} has the form of \eqref{eq:hjb}.

\subsection{Subvalue Functional}

The solution of program \eqref{eq:u_cont} can create a lower-bound on the value functional $V^*$.

\subsubsection{Properties of Subvalue Functionals}

\begin{defn}
\label{def:subvalue}
A \textit{subvalue functional} is a functional $\mathcal{V}(t, x, w)$ such that
\begin{align}
   \mathcal{V}(t, x, w) &\leq V^*(t, x, w)  & & \forall t \in [0, T], x \in X, \ w \in PC([-\tau, 0], X).
\end{align}
\end{defn}

\begin{thm}
    Any functional $\mathcal{V}(t, x, x_\tau)$ with derivative $\dot{\mathcal{V}}(t, x, x_\tau, u)$ that satisfies the following two properties is a subvalue functional for $V^*$:
    \begin{subequations}        
    \begin{align}
        &J_T(x) - \mathcal{V}(T, x, w) \geq 0 & & \forall x \in X_T, \ w \in PC([-\tau, 0], X) \label{eq:hjb_term}\\
        &J(t, x, u) + \dot{\mathcal{V}}(t, x, w, u) \geq 0, & & \forall t \in T, \ x \in X, \ w \in PC([-\tau, 0], X), \ u \in U. \label{eq:hjb_lie}
    \end{align}
    \end{subequations}
\end{thm}

\begin{proof}
This result follows by following the steps of Proposition 1's proof from \cite{jones2021polynomial}.

Let $\tilde{u} \in \mathcal{U}$ be an arbitrary control policy starting at the initial condition $(t_0, z, w)$, resulting in a trajectory $\tilde{x}(t)$. Denote $\tilde{x}_\tau(\cdot)$ as the history function $\tilde{x}_t(t') = \tilde{x}(t + t') \ \forall t' \in [-\tau, 0]$.

Relation \eqref{eq:hjb_lie} ensures that for all $t \in [t_0, T]$,
\begin{equation}
    \dot{\mathcal{V}}(t, \tilde{x}(t), \tilde{x}_t(\cdot), \tilde{u}(t)) + J(t, \tilde{x}(t), \tilde{u}(t)) \geq 0.
\end{equation}
Integrating the above term with respect to $t$ yields
\begin{subequations}
\begin{align}
    0 &\leq \int_{t=t_0}^T \dot{\mathcal{V}}(t, \tilde{x}(t), \tilde{x}_t(\cdot), \tilde{u}(t)) + J(t, \tilde{x}(t), \tilde{u}(t))dt \\
    0&\leq \underbracket{\mathcal{V}(T, \tilde{x}(T), \tilde{x}_T(\cdot))}_{V(T)}- \underbracket{\mathcal{V}(t_0, \tilde{x}(t_0), \tilde{x}_{t_0}(\cdot))}_{\mathcal{V}(t_0)} +  \int_{t=t_0}^T
    J(t, \tilde{x}(t), \tilde{u}(t))dt \\
    \mathcal{V}(t_0) &\leq \mathcal{V}(T) + \int_{t=t_0}^T
    J(t, \tilde{x}(t), \tilde{u}(t))dt \label{eq:subvalue_vt1}\\
    \mathcal{V}(t_0) &\leq J_T(\tilde{x}(T)) + \int_{t=t_0}^T
    J(t, \tilde{x}(t), \tilde{u}(t))dt. \label{eq:subvalue_jt1}
\end{align}
\end{subequations}

The transformation of \eqref{eq:subvalue_vt1} to  \eqref{eq:subvalue_jt1} follows from relations \eqref{eq:value_term} and \eqref{eq:u_cont_term} $(\mathcal{V}(T) \leq J(\tilde{x}(T))$. When $\tilde{u}$ is a minimizing control  $u^*$ (if it exists), then the right-hand side of \eqref{eq:subvalue_jt1} is the optimal value functional $V^*(t_0, z, w(\cdot))$ and the left-hand side is $\mathcal{V}(t_0, z, w(\cdot))$. The proof that $\mathcal{V}$ is a lower bound on $V^*$ is therefore complete, given that $\mathcal{V}(t_0, z, w(\cdot)) \leq V^*(t_0, z, w(\cdot))$ will hold for all choices of $(t_0, z, w(\cdot))$.
\end{proof}


\subsubsection{Recovery of a Subvalue Functional}
The dual solution $(v, \phi)$ from \eqref{eq:u_cont} may be assembled into a functional,
\begin{equation}
\label{eq:subvalue}
    V(t, x, z(\cdot)) = v(t, x) + \int_{t}^{\min(t+\tau, T)} \phi(s, z(s-\tau)) ds.
\end{equation}

The bias term $\phi$ in \eqref{eq:u_cont_phi} is defined and is $C^0$ continuous only between times $t \in [0, T]$. If the hard integration limit at $T$ was not present, then $\phi$ would be queried at undefined values $t \in (T, T + \tau]$. 
The terminal value of the value functional is
\begin{equation}
    V(T) = V(t, x(T), x([T-\tau, T])) = v(T, x(T)) + \int_{T}^{T} \phi_i(s, x(s-\tau))ds = v(T, x(T)). \\ \label{eq:value_term}
\end{equation}

The objective in \eqref{eq:u_cont_obj} is the evaluation of the value functional at time $t=0$ along the optimal controlled trajectory $x^*(t)$:
\begin{equation}
    V(0) = V(0, x_h(0), x_h) = v(0, x_h(0)) + \int_{0}^{\tau} \phi_i(s, x_h(s-\tau))ds. \label{eq:value_init}
\end{equation}

The time-derivative (co-invariant derivative) of the value functional $\dot{V}$ is
\begin{align}
\label{eq:v_dot}
    \dot{V}(t, z, w(\cdot), u) &= \Lie_f v(t, x(t)) + I_{[0, T-\tau]}(t) \phi(t+\tau, z) -\phi(t, w(- \tau)).
\end{align}

\begin{thm}
    The functional \eqref{eq:subvalue} is a subvalue functional in the sense of \eqref{def:subvalue}.
\end{thm}
\begin{proof}
    The terminal constraint \eqref{eq:u_cont_term} satisfies \eqref{eq:hjb_term} given the terminal value evaluation in \eqref{eq:value_term}. The combination of \eqref{eq:u_cont_f} and \eqref{eq:u_cont_f_end} together satisfy \eqref{eq:hjb_lie} under the derivative value expression in \eqref{eq:v_dot}.
\end{proof}





\subsubsection{Continuity of the Recovered Value Functional}
\label{sec:continuity_subvalue}
Let $u^*(t) \in \mathcal{U}$ be the optimal (infimizing) trajectory of problem \eqref{eq:u_traj} given $z, w(\cdot)$, inducing a controlled trajectory $x^*(t) = x(t \mid z, w(\cdot))$.
The value functional evaluated along the optimal trajectory $V^*(t) = V(t, x^*(t), x^*([t-\tau, t])$ is $C^0$-continuous in $t$, but is not necessarily $C^1$-continuous in $t$. At the time $t=T-\tau$, define the value functional evaluations
\begin{align}
    V_-^* &= \lim_{t \rightarrow (T - \tau)^-} V^*(t) & V_+^* &= \lim_{t \rightarrow (T - \tau)^+} V^*(t).
\end{align}

The difference in these evaluations for every lag $i^*$ is
\begin{align}
   \Delta V^* = V_-^*-V_+^* &= \left(\int_{T - \tau^-}^{T^-} - \int_{T - \tau^+}^{T^+}\right) \phi(s, x(s-\tau_{i})) ds = 0.
\end{align}

The value functional $V^*(t)$ is therefore a member of $C^0([0, T])$ given that $v \in C^1([0, T] \times X)$ and $\phi \in C^0([0, T] \times X)$.

The value functional derivative evaluations are
\begin{align}
    \dot{V}_-^* &= \lim_{t \rightarrow (T - \tau)^-} \dot{V}^*(t) & \dot{V}_+^* &= \lim_{t \rightarrow (T - \tau_{i})^+} \dot{V}^*(t).
\end{align}

The difference in the derivative evaluations on both sides of $t=T-\tau$ is
\begin{align}
   \Delta \dot{V}^* = \dot{V}_-^*-\dot{V}_+^* &= \phi(T, x(T -\tau)).
\end{align}

It is not guaranteed that $\phi(T, x(T -\tau)) = 0$, so the value functional $V^*$ may have discontinuous first derivatives. The value functional is therefore $C^0$ in time along trajectories, and fails to be $C^1$ at the time $t=T-\tau$.

We form an additional conjecture to \ref{conj:delay} in the \ac{OCP} case based on the tightness conditions in \cite{lewis1980relaxation}.
\begin{conj}
Assume for the purposes of this conjecture that:
\begin{itemize}
    \item[A1'] The sets $\{[-\tau, T], X, X_T, U\}$ are all compact.
    \item[A2'] The costs $J, J_T$ are continuous.
    \item[A3'] The dynamics $f(t, x_0, x_1, u)$ are Lipschitz in their arguments.
    \item[A4'] The history $x_h$ is inside $PC([-\tau, 0], X)$.
    \item[A5'] The image of $f(t, x_0, X, U)$ is convex for each fixed $(t, x_0)$.
    \item[A6'] The mapping $v \mapsto \inf_{u \in U} J(t, x_0, x_1, u): f(t, x_0, x_1, u) = v$ is convex in $v \in \R^n$.
\end{itemize}
Then there is no relaxation gap between \eqref{eq:u_traj} and \eqref{eq:u_cont} ($P^*=p^*$).
\end{conj}

\subsection{Approximate Recovery}
A control policy $u(t)$ may be extracted from the value functional $V$ through the trajectory condition \eqref{eq:hjb_traj_true} by
\begin{subequations}
\begin{align}
    \label{eq:u_recovery}
    u(t) &= \argmin_u  \Lie_f v(t, x(t)) +  I_{[0, T-\tau]}(t) \phi(t+\tau, x(t))* - \phi(t, x(t - \tau)) + J(t, x(t), u(t))\\
    &=\argmin_u  f(t, x(t),  x(t - \tau), u) \cdot \nabla_{x} 
    v(t, x(t)) + J(t, x(t), u).
\end{align}
\end{subequations}

The work in \cite{jones2021polynomial} quantifies performance bounds of polynomial value function approximations for \ac{ODE} systems in terms of $W^1$ Sobolev norms away from the true value function. Quantifying performance bounds of the extracted controller of \eqref{eq:u_recovery} is an open problem.

\subsection{Example of Optimal Control}
\label{sec:u_example}

An example of optimal control is presented on the one-dimensional linear system:
\begin{align}
x'(t) &= -3x(t) - 5x(t-0.25) + u & &\forall t \in [0, 1] \nonumber \\
x_h(t) &= -1 & & \forall t \in [-0.25, 0].
\end{align}
This system has one lag with $\tau = 0.25$ and a time horizon of $T = 1$. The state and control constraints are $X = [-1, 1]$ and $U = [-1, 1]$. With a control weight of $R = 0.01$, the penalties are
\begin{equation}
    J(t, x, u) = 0.5 x^2 + 0.5 R u^2 \qquad J_T(x) = 0.\end{equation}

The open loop total cost is $0.0674$. Table \ref{tab:linear_bounds} lists optimal control value approximations for this system.
\begin{table}[h]
    \centering
    \caption{\label{tab:linear_bounds}\ac{SDP} approximation bounds to program \eqref{eq:u_meas}}
    \begin{tabular}{c|c c c c c c}
        order&	1&	2	&3	&4	&5	&6  \\
        bound	&7.90E-05	&0.0322&	0.0386&	0.0391	&0.0393	&0.0393 
    \end{tabular}
\end{table}

The applied control $u(t)$ may be recovered through equation \eqref{eq:u_recovery} as

\begin{equation}
    u(t) = \text{Saturate}_{[-1,1]} \left( -\frac{1}{R}\partial_x v(t,x(t))  \right).
\end{equation}

The trajectories and nonnegative functions are plotted for order 4 in Figures \ref{fig:traj}-\ref{fig:aux_value}. The order 4 control bound is 0.0391, and the cost evaluated along the controlled trajectory is 0.0394. 

\begin{figure}[h]
    \centering
    \includegraphics[width=0.7\linewidth]{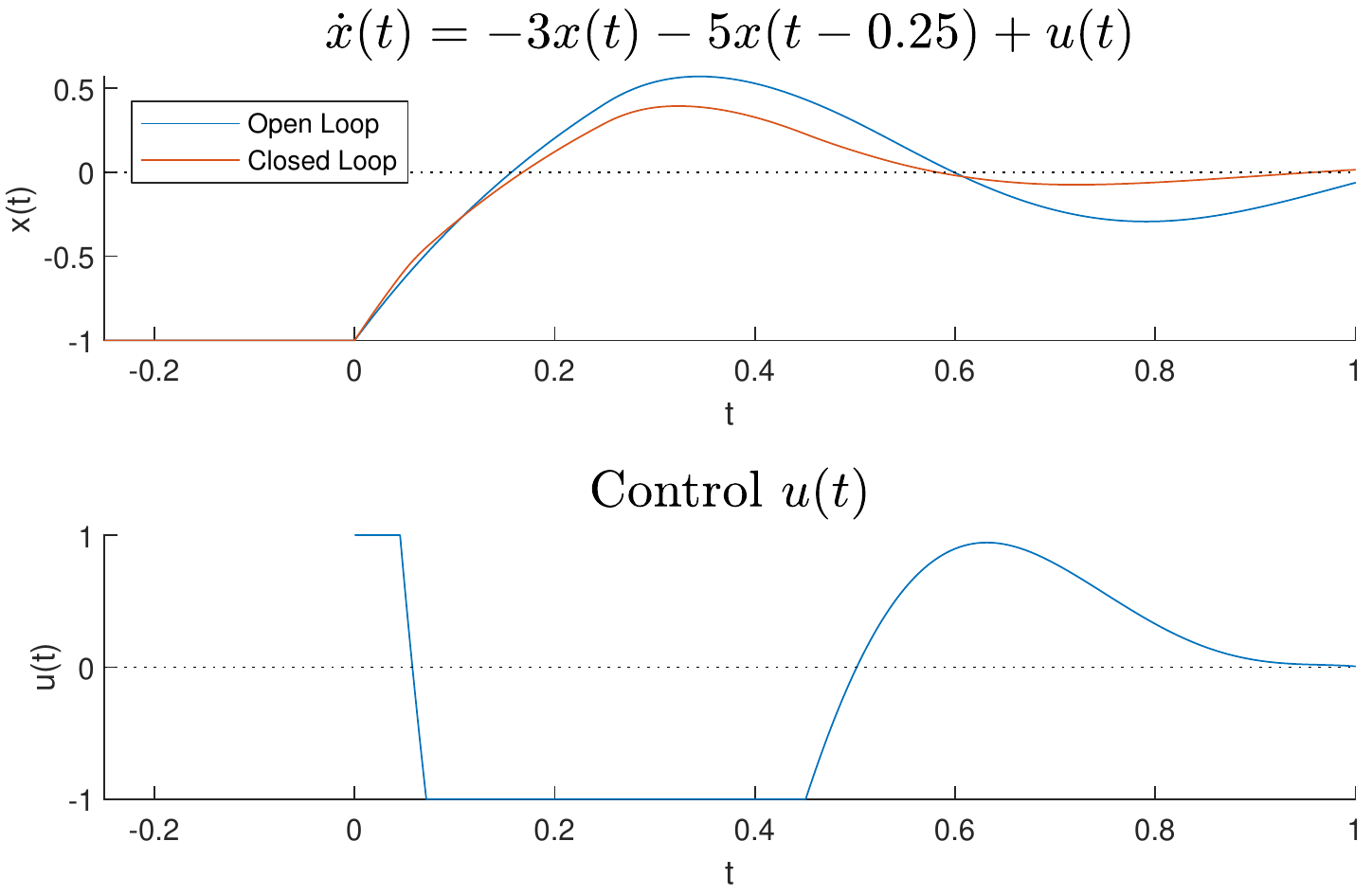}
    \caption{Open and closed loop trajectory with control}
    \label{fig:traj}
\end{figure}

\begin{figure}[!h]
    \centering
    \includegraphics[width=\linewidth]{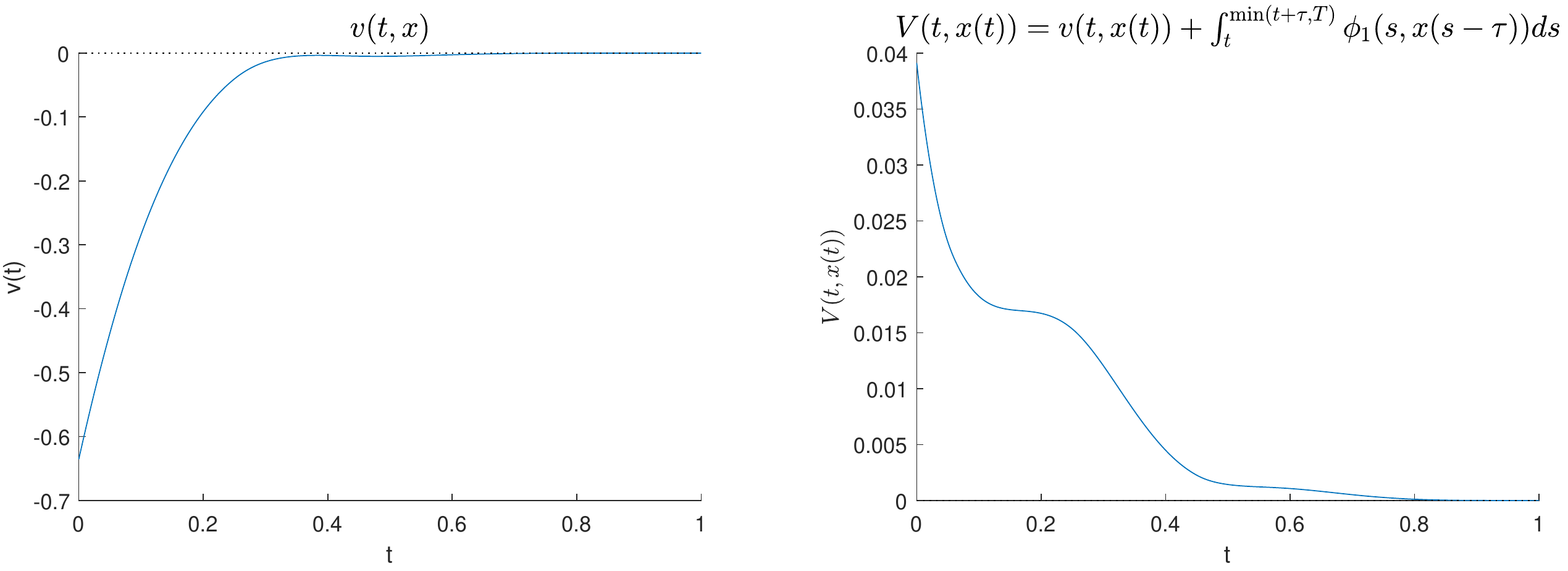}
    \caption{Auxiliary function $v$ and value functional $V$ from \eqref{eq:subvalue}}
    \label{fig:aux_value}
\end{figure}

\section{Improved Accuracy of Control Problems}
\label{app:more_accurate}

This appendix lays out methods to reduce the conservatism of \ac{DDE} \acp{OCP} from appendix \ref{app:ocp} by adding new infintie-dimensional nonnegativity constraints. 

All approaches discussed in this appendix may be applied to peak estimation, but are presented here in the simplified fixed-terminal-time single-history \ac{OCP} setting.

\subsection{Spatial Partitioning}

The constraints \eqref{eq:u_cont_term}-\eqref{eq:u_cont_f_end} must hold in support sets defined by $[0, T] \times X \times U$. Assume that there exists a decomposition of the state  spaces $X = \cup_k X_k$ such that $\forall k: \textrm{dim}(X_k) = n$ and  $ \forall k, k': \textrm{int}(X_k \cap X_{k'}) = \varnothing$ (cells $X_k$ are full-dimensional and their intersections are not full dimensional). Further assume a similar decomposition exists for the control set 
 $U = \cup_\ell U_k$.
 
Let $(v_k(t, x), \phi_k(t,x))$ be functions associated with each space $X^k$. A space-control partition of \eqref{eq:u_cont} is:
\begin{subequations}
\label{eq:u_cont_space}
\begin{align}
    d^* = & \ \sup \quad \sum_k \left(I_{X_k}(x_h(0)) v_k(0, x_h(0)) \right)+ \textstyle\int_{-\tau}^{0} \phi(t + \tau, x_h(t))  dt\label{eq:u_cont_obj_space}& \\
    & \forall x \in X_k: \nonumber \\
    & \qquad J_T(x) - v_k(T, x) \geq 0 \label{eq:u_cont_term_space}  \\
    & \forall (t, x_0, x_1 ,u) \in [0, T-\tau]  X_k \times X_{k'} \times U_\ell: \nonumber \\
    & \qquad \Lie_f v_k + J(t, x_0, u) -  \phi_{k'}(t, x_1)  +  \phi_k(t+\tau, x_0) \geq 0 \label{eq:u_cont_f_space}  \\
    & \forall (t, x_0, x_1 ,u) \in [T-\tau, T] \times X_k \times X_{k'} \times U_\ell: \nonumber \\        
    & \qquad  \Lie_f v_k + J(t, x_0, u) -  \phi_{k'}(t, x_1) \geq 0 \label{eq:u_cont_f_end_space} & &  \\
    & \forall (t, x) \in [0, T] \times (X_k \cap X_{k'}): \nonumber \\
    & \qquad v_k(t, x) = v_{k'}(t, x) \label{eq:u_cont_agreement_space} \\
    &\forall k: v_k \in C^1([0, T] \times X_k)  \label{eq:u_cont_v_space} \\
    &\forall k: \phi_k \in C([0, T] \times X_k)  \label{eq:u_cont_phi_space}.
\end{align}
\end{subequations}

The $v_k$ terms agree on boundary regions between state cells by \eqref{eq:u_cont_agreement_space}. The $\phi_k$ terms remain continuous (bounded measurable), but this partitioning has an impact when evaluating the finite-degree  \acp{SDP}.  

\subsection{Temporal Partitioning}

We utilize the following lemma to provide conditions for temporal partitioning:
\begin{lem}
\label{lem:suff_int}
A sufficient condition for $\int_{t_0}^{t_1} g_1(t) dt \geq \int_{t_0}^{t_1} g_2(t) dt$ is that \begin{equation}
    \forall t \in [t_0, t_1]: g_1(t) \geq g_2(t).
\end{equation}
\end{lem}
Define the following time breaks (partition) arranged in sorted order as
\begin{equation}
    T_{break} = \{0, t_1, \ldots, t_{k-1}, t_k=T-\tau, t_{k+1}, \ldots, t_{k+\ell-1}, t_{k+\ell} = T\}.
\end{equation}

Let $[t_{b}, t_{b+1}]$ and $[t_{b-1}, t_{b}]$ be regions in $T_{break}$. The subvalue functionals from \eqref{eq:subvalue} defined in this region must satisfy 
\begin{subequations}
\label{eq:subvalue_break_decrease}
\begin{align}
    V_b(t_b, x, z(\cdot)) &\geq V_{b-1}(t_b, x, z(\cdot)) \\
    v_b(t_b, x) + \int_{t_b}^{\min(t_b+\tau, T)} \phi_b(s, z(s-\tau)) ds & \geq v_{b-1}(t_b, x) + \int_{t_b}^{\min(t_b+\tau, T)} \phi_{b-1}(s, z(s-\tau)) ds.
\end{align}
\end{subequations}

The sufficient condition in Lemma \ref{lem:suff_int} may be used to accomplish this relation in \eqref{eq:subvalue_break_decrease}, ensuring that the subvalue function will always decrease when traversing a time break:
\begin{subequations}
\label{eq:time_suff}
\begin{align}
    & v_b(t_b, x) \geq v_{b-1}(t_b, x) & & \forall x \in X\\
    & \phi_b(t, x) \geq \phi_{b-1}(t, x) & & \forall (t, x) \in [t_b, \min(t_b + \tau, T)] \times X.
\end{align}
\end{subequations}

The resultant time-partitioned \ac{LP} is,
\begin{subequations}
\label{eq:u_cont_time}
\begin{align}
    d^* = & \ \sup \quad v(0, x_h(0)) +\textstyle \textstyle\int_{-\tau}^{0} \phi_{k+\ell}(t + \tau, x_h(t)) dt\label{eq:u_cont_obj_time}& \\
    & \forall x \in X:  \nonumber \\
    &\qquad J_T(x) - v_{k+\ell}(T, x) \geq 0 \label{eq:u_cont_term_time} \\
    & \forall (t, x_0, x_1 ,u) \in [t_{k'}, t_{k'+1}] \times X^2 \times U, k'=0..k-1: \nonumber \\
    & \qquad  \Lie_f v_{k'} + J(t, x_0, u) -  \phi_{k'}(t, x_1)  +  \phi_{k'}(t+\tau, x_0) \geq 0 \label{eq:u_cont_f_time} \\
    &  \forall (t, x_0, x_1 ,u) \in [t_{k'}, t_{k'+1}] \times X^2 \times U, k'=k..k+\ell: \nonumber \\
    & \qquad \Lie_f v_{k'} + J(t, x_0, u) -  \phi_{k'}(t, x_1) \geq 0 \label{eq:u_cont_f_end_time} \\
    & \forall x \in X,  k'=1..k+\ell-1: \nonumber \\
    & \qquad  v_{k'}(t_{k'}, x) \leq v_{k'+1}(t_{k'}, x)  \\
    & \forall (t, x) \in [t_{k'}, \min(t_{k'} + \tau, T)] \times X, k'=1..k+\ell-1 \nonumber\\
    &  \qquad  \phi_{k'}(t, x) \leq \phi_{k'+1}(t, x)\\
    & \forall k'=0..k+\ell: \\
    &\qquad v_{k'} \in C^1([t_{k'}, t_{k'+1}] \times X) & & \label{eq:u_cont_v_time} \\
    &\qquad \phi_{k'} \in C([t_{k'}, \min(t_{k'+1}+\tau, T)] \times X).  \label{eq:u_cont_phi_time}
\end{align}
\end{subequations}

\subsection{Double Integral Functionals}

The subvalue functional in \eqref{eq:subvalue} has a single integral term for each delay. 
Some Lyapunov-Krasovskii or Barrier methods for \ac{DDE}  analysis employ double integrals, such as the following functional for a single delay $\tau$ \cite{papachristodoulou2005tutorial},
\begin{align}
    V(t, z, w) &= v(t, z) + \int_{t}^{\min(t+\tau, T)} \phi_i(s, w(s-\tau-t))ds \nonumber \\
    &+ \int_{t}^{\min(t+\tau, T)}\int_{-\tau}^{0}\psi(s, q, w(s-\tau-t), w(q))dq ds. \label{eq:subvalue_double}
\end{align}

The time derivative of \eqref{eq:subvalue_double} in the time span $t \in [0, T-\tau)$ is
\begin{align}
        \dot{V}(t,z,w) &= \Lie_f v(t, z) + \phi(t+\tau, w(0)) - \phi(t, w(-\tau))\\ &+\textstyle\int_{-\tau}^{0}\psi(t+\tau, q, w(0), w(q))dq \nonumber\\
    &-\textstyle\int_{-\tau}^{0}\psi(t, q, w(-\tau), w(q))dq, \label{eq:double_span_start}
\end{align}
and between $t \in (T-\tau, 0]$ is
\begin{align}
    \dot{V}(t,z,w) &= \Lie_f v(t, z)  - \phi(t, w(-\tau)) -\textstyle\int_{-\tau}^{0}\psi(t, q, w(-\tau), w(q))dq. \label{eq:double_span_end}
\end{align}

The derivative $\dot{V}$ has a discontinuity present at $t=T-\tau$, just as described in Section \eqref{sec:continuity_subvalue}.

A sufficient condition for the inequality \eqref{eq:hjb_lie} to be fulfilled is that the following functions associated with \eqref{eq:double_span_start} and  \eqref{eq:double_span_end}  (moving all terms under the $dq$ integral) are nonnegative:
\begin{align}
&\forall t\in [0, T-\tau], (x_0, x_1, \tilde{x}) \in X^3, q \in [-\tau, 0]: \nonumber \\
    &\qquad \tau^{-1}\left(\Lie_f v(t, z) + J(t, x_0, u) + \phi(t+\tau, x_0) - \phi(t, x_1) \right) \nonumber \\
    &\qquad + \psi(t+\tau, q, x_0, \tilde{x}) - \psi(t, q, x_0, \tilde{x}) \geq 0  \label{eq:double_nonneg_start} \\
&\forall t\in [T-\tau, T], (x_0, x_1, \tilde{x}) \in X^3, q : 
 \nonumber \\
    &\qquad \tau^{-1} \left(\Lie_f v(t, z) + J(t, x_0, u)  - \phi(t, x_1)\right) - \psi(t, q, x_0, \tilde{x}) \geq 0. \label{eq:double_nonneg_end}
\end{align}

Lemma \ref{lem:suff_int} is utilized to enforce nonnegativity of the integral terms in \eqref{eq:double_span_start} and \eqref{eq:double_span_end}. The $\tau^{-1}$ scale factor arises from placing a $q$-independent term (such as $J(t, x_0, u)$) inside the integral. The variable $q \in [-\tau, 0]$ is the integrated (swept) time, and $\tilde{x} \in X$ abstracts out the swept state $w(q)$. The dual formulation of constraints \eqref{eq:double_nonneg_start} and \eqref{eq:double_nonneg_end} involve occupation measures $\bar{\mu}_0 \in \Mp{[0, T-\tau] \times X^3 \times U}$ and $\bar{\mu}_1 \in \Mp{[T-\tau, T] \times X^3 \times U}$. This construction may be generalized to \acp{DDE} with $r$ delays by adding a double-integral term for each delay.

\section{Joint+Component Measure}
\label{app:joint_component}

This appendix details an alternate notion of \ac{MV}-solutions for \acp{DDE}. Solving peak estimation problems through these methods will return more conservative but quicker-executing 
 programs (computationally) as compared to the results in Section \ref{sec:peak_lp}.

\subsection{Measure Program}
The \ac{MV}-solution involves a joint occupation measure $\bmu$ and component measures $\omega_0, \omega_1$:
\begin{subequations}
\label{eq:weak_solution_jc}
    \begin{align}
        &\textrm{History}  & \mu_h &\in \Mp{H_0} \label{eq:weak_solution_jc_history} \\
       & \textrm{Initial} & \mu_0 &\in \Mp{X_0} \\
       &\textrm{Peak}  &\mu_p &\in \Mp{[0, T] \times X} \\
       &\textrm{Time-Slack} & \nu &\in \Mp{[0, T] \times X} \label{eq:weak_solution_jc_slack} \\
        &\textrm{Joint Occupation}  & \bmu &\in \Mp{[0, T] \times X^2} \label{eq:weak_solution_jc_occ} \\        
        &\textrm{Component Start }  & \omega_0 &\in \Mp{[0, T-\tau] \times X} \\
        &\textrm{Component End }  & \omega_1 &\in \Mp{[T-\tau, T] \times X}.
    \end{align}
\end{subequations}

The peak estimation \ac{LP} for the Joint+Component framework is
\begin{subequations}
\label{eq:peak_delay_jc_meas}
    \begin{align}
        p^* = & \ \sup \quad \inp{p}{\mu_p} \label{eq:peak_delay_jc_obj} \\
    & \inp{1}{\mu_0} = 1 \label{eq:peak_delay_jc_prob}\\    
    & \pi^{t}_\# \mu_h = \lambda_{[-\tau, 0]} \label{eq:peak_delay_jc_hist}\\   
    & \mu_p = \delta_0 \otimes\mu_0 + \pi^{t x_0}_\# \Lie_f^\dagger \bmu \label{eq:peak_delay_jc_flow}\\   
    & \pi^{t x_0}_\# \bmu = \omega_0 + \omega_1 \label{eq:peak_delay_jc_cons_start}\\ 
    & \pi^{t x_1}_\# \bmu + \nu = S^\tau_\#(\mu_h + \omega_0) \label{eq:peak_delay_jc_cons_end}\\ 
    & \textrm{Measure Definitions from  \eqref{eq:weak_solution_jc}.} \label{eq:peak_delay_jc_def}
    \end{align}
\end{subequations}

The history-validity and Liouville constraints in \eqref{eq:peak_delay_jc_meas} are the same as in \eqref{eq:peak_delay_meas} under the relation $\bmu = \bmu_0 + \bmu_1$. The consistency constraint in the Joint+Component formulation is split up into the pair \eqref{eq:peak_delay_jc_cons_start}-\eqref{eq:peak_delay_jc_cons_end}.

\begin{thm} 
Program \eqref{eq:peak_delay_jc_meas} returns an upper bound on \eqref{eq:peak_delay_traj}.
\end{thm}
\begin{proof}
Let $x_h \in \mathcal{H}$ be a history that generates the trajectory $x(t \mid x_h)$, and let $t^* \in [0, T]$ be a stopping time. 

Just as in Theorem \ref{thm:delay_upper_bound}, measures can be picked as $\mu_0 = \delta_{x = x_h(0^+ \mid x_h)}$, $\mu_p = \delta_{t=t^*} \otimes \delta_{x=x(t^* \mid x_h)}$, $\mu_h$ as the occupation measure of $x_\xi(t)$ in the times $[-\tau, 0]$, and $\bmu$ as the occupation measure of $z(t) = (x(t \mid x_h), x(t-\tau \mid x_h))$ in the times $[0, t^*]$.

When $t^* \in [0, T-\tau]$, then $\omega_0$ is the occupation measure of $x(t \mid x_h)$ in times $[0, t^*]$, $\omega_1$ is the zero measure, and $\nu$ is the occupation measure of $x(t-\tau \mid x_h)$ in times $[t^*, t^*+\tau]$. When $t^* \in (T-\tau, T]$, then $\omega_0$ is the occupation measure of $x(t \mid x_h)$  in the  times $[0, T-\tau]$, $\omega_1$ is the occupation measure of $x(t \mid x_h)$ in the times $[T-\tau, t^*$, and $\nu$ is the occupation measure of $x(t-\tau \mid x_h)$ in the times $[T-\tau, T]$.

All measures inside \eqref{eq:weak_solution_jc} have been defined for a valid trajectory, proving that \eqref{eq:peak_delay_jc_meas} upper-bounds \eqref{eq:peak_delay_traj}.
\end{proof}

\begin{thm}
    The Joint+Component measure program \eqref{eq:peak_delay_jc_meas} is also an upper bound on \eqref{eq:peak_delay_meas}.
\end{thm}
\begin{proof}
    Let $(\mu_h, \mu_0, \mu_p, \nu, \bmu_0, \bmu_1)$ be a feasible set of measures for the constraints of \eqref{eq:peak_delay_meas}.

    After performing the following definitions,
    \begin{align}
        \bmu &= \bmu_0 + \bmu_1 & \omega_0 &= \pi^{t x_0} \bmu_0 & \omega_1 &= \pi^{t x_0} \bmu_1,
    \end{align}
    the measures $(\mu_h, \mu_0, \mu_p, \nu, \bmu, \omega_0, \omega_1)$ are feasible solutions for the constraints of \eqref{eq:peak_delay_jc_meas}.    
\end{proof}

Note how the Joint+Component \ac{MV}-solution involves only one measure involving $(t, x_0, x_1)$ together ($\bmu$ in \eqref{eq:weak_solution_jc_slack}), while the solution in \eqref{eq:weak_solution} has two measures ($\bmu_0, \bmu_1$). Application of the moment-\ac{SOS} hierarchy towards solving problems in \eqref{eq:weak_solution_jc_slack} result in only one Gram matrix of maximal size $\binom{1+2n+d}{d}$.

\subsection{Function Program}

The gap between \eqref{eq:peak_delay_jc_meas} and \eqref{eq:peak_delay_meas} can most easily be observed by examining the dual program of  \eqref{eq:peak_delay_jc_meas}:

\begin{subequations}
\label{eq:peak_delay_jc_cont}
\begin{align}
    d^* = & \ \inf_{\gamma \in \R} \gamma + \textstyle \int_{t=-\tau}^0 \xi(t) dt & & \\
    & \gamma \geq v(0, x)  & & \forall x \in  X_0 \label{eq:peak_delay_jc_cont_init}\\
    & v(t, x) \geq p(x) & & \forall (t, x) \in [0, T] \times X \label{eq:peak_delay_jc_cont_p} \\
    & \xi(t) + \phi_1(t+\tau, x) & & \forall (t, x) \in H_0 \\
    & 0 \geq \Lie_f v(t, x_0) + \phi_0(t, x_0) + \phi_1(t, x_1) & & \forall (t, x_0, x_1)\in [0, T] \times X^{2} \label{eq:jc_omega_Lie} \\
    & \phi_1(t, x) \leq 0 & & \forall (t, x) \in [0, T] \\ 
    & \phi_0(t, x) + \phi_1(t+\tau, x) \geq 0 & & \forall (t, x) \in [0, T-\tau] \times X \label{eq:jc_omega_0} \\
    & \phi_0(t, x)\geq 0 & & \forall (t, x) \in [T-\tau, T] \times X \label{eq:jc_omega_1} \\
    & v(t,x) \in C^1([0, T]\times X) \label{eq:peak_delay_jc_cont_v}& & \\
    &\phi_0(t, x), \phi_1(t, x) \in C([0, T] \times X)  \label{eq:peak_delay_jc_cont_phi} \\
    &\xi(t) \in C([-\tau, 0]).
\end{align}
\end{subequations}

Adding together \eqref{eq:jc_omega_Lie} and \eqref{eq:jc_omega_0} yields constraint \eqref{eq:peak_delay_cont_lie0} in $[0, T-\tau] \times X^2$. Similarly, the addition of \eqref{eq:jc_omega_Lie} and \eqref{eq:jc_omega_1} forms constraint \eqref{eq:peak_delay_cont_lie1}. The dual formulation in \eqref{eq:peak_delay_jc_cont} enforces nonnegativity of addends inside whole-terms of \eqref{eq:peak_delay_meas}. The constraints of \eqref{eq:peak_delay_jc_cont} are stricter than of \eqref{eq:peak_delay_cont}, resulting in a lowered infimum/upper bound on peak value.

\subsection{Joint+Component Example}

Table \eqref{tab:jc_compare_flow} compares moment-\ac{SOS} \acp{SDP} associated to programs \eqref{eq:peak_delay_meas} and \eqref{eq:peak_delay_jc_meas} for the delayed Flow example in Section \ref{ex:delay_flow}.

\begin{table}[h]
\caption{\label{tab:jc_compare_flow}Comparison of \eqref{eq:peak_delay_meas} and \eqref{eq:peak_delay_jc_meas} \ac{SDP} bounds for the delayed Flow system}
\centering
\begin{tabular}{llllll}
degree $d$               & 1    & 2      & 3      & 4    & 5  \\ \hline
Joint+Component \eqref{eq:peak_delay_jc_meas} & 1.25 & 1.223  & 1.1937 & 1.1751 & 1.1636\\
Standard \eqref{eq:peak_delay_meas}  & 1.25 & 1.2183 & 1.1913 & 1.1727 & 1.1630 
\end{tabular}
\end{table}

\begin{table}[h]
\caption{\label{tab:jc_compare_flow_time}Time (seconds) to obtain \ac{SDP} bounds in Table \ref{tab:jc_compare_flow}}
\centering
\begin{tabular}{llllll}
degree $d$  & 1    & 2      & 3      & 4    & 5  \\ \hline
Joint+Component \eqref{eq:peak_delay_jc_meas} & 0.782 & 0.991  & 5.271 & 31.885& 336.509\\
Standard \eqref{eq:peak_delay_meas}  & 0.937 & 1.190 & 9.508 & 105.777 & 552.496
\end{tabular}
\end{table}

\end{document}